\documentclass[reqno,11pt,a4paper]{amsart}
\usepackage[a4paper,left=35mm,right=35mm,top=35mm,bottom=35mm,marginpar=25mm]{geometry} 
\usepackage{amsmath}
\usepackage{amssymb}
\usepackage{amsthm}
\usepackage{amscd}
\usepackage{esint}
\usepackage[ansinew]{inputenc}
\usepackage{cite}
\usepackage{bbm}
\usepackage{xcolor}
\usepackage[final]{graphicx}
\usepackage[colorlinks=true,citecolor=blue,linkcolor=blue]{hyperref}
\usepackage{calc}
\usepackage{bm}
\usepackage{enumerate}
\usepackage{transparent}
\usepackage{stmaryrd}
\usepackage{mathtools}

\numberwithin{equation}{section}

\newtheoremstyle{thmlemcorr}{10pt}{10pt}{\itshape}{}{\bfseries}{.}{10pt}{{\thmname{#1}\thmnumber{ #2}\thmnote{ (#3)}}}
\newtheoremstyle{thmlemcorr*}{10pt}{10pt}{\itshape}{}{\bfseries}{.}\newline{{\thmname{#1}\thmnumber{ #2}\thmnote{ (#3)}}}
\newtheoremstyle{defi}{10pt}{10pt}{\itshape}{}{\bfseries}{.}{10pt}{{\thmname{#1}\thmnumber{ #2}\thmnote{ (#3)}}}
\newtheoremstyle{remexample}{10pt}{10pt}{}{}{\bfseries}{.}{10pt}{{\thmname{#1}\thmnumber{ #2}\thmnote{ (#3)}}}
\newtheoremstyle{ass}{10pt}{10pt}{}{}{\bfseries}{.}{10pt}{{\thmname{#1}\thmnumber{ #2}\thmnote{ (#3)}}}

\theoremstyle{thmlemcorr}
\newtheorem{theorem}{Theorem}
\numberwithin{theorem}{section}
\newtheorem{lemma}[theorem]{Lemma}
\newtheorem{corollary}[theorem]{Corollary}
\newtheorem{proposition}[theorem]{Proposition}

\newtheorem{conjecture}[theorem]{Conjecture}

\theoremstyle{thmlemcorr*}
\newtheorem{theorem*}{Theorem}
\newtheorem{lemma*}[theorem]{Lemma}
\newtheorem{corollary*}[theorem]{Corollary}
\newtheorem{proposition*}[theorem]{Proposition}
\newtheorem{problem*}[theorem]{Problem}
\newtheorem{conjecture*}[theorem]{Conjecture}

\theoremstyle{defi}
\newtheorem{definition}[theorem]{Definition}

\theoremstyle{remexample}
\newtheorem{remark}[theorem]{Remark}
\newtheorem{example}[theorem]{Example}

\theoremstyle{ass}

\newcommand{\Crm}{\mathrm{C}}

\newcommand{\Lrm}{\mathrm{L}}

\newcommand{\Wrm}{\mathrm{W}}

\newcommand{\Acal}{\mathcal{A}}
\newcommand{\Bcal}{\mathcal{B}}
\newcommand{\Ccal}{\mathcal{K}}
\newcommand{\Dcal}{\mathcal{D}}
\newcommand{\Ecal}{\mathcal{E}}
\newcommand{\Fcal}{\mathcal{F}}

\newcommand{\Hcal}{\mathcal{H}}

\newcommand{\Kcal}{\mathcal{K}}
\newcommand{\Lcal}{\mathcal{L}}
\newcommand{\Mcal}{\mathcal{M}}

\newcommand{\Rcal}{\mathcal{R}}
\newcommand{\Scal}{\mathcal{S}}

\newcommand{\Ybf}{\mathbf{Y}}

\newcommand{\Abb}{\mathbb{A}}
\renewcommand{\Bbb}{\mathbb{B}}

\newcommand{\Lbb}{\mathbb{L}}

\newcommand{\Nbb}{\mathbb{N}}

\newcommand{\Sbb}{\mathbb{S}}
\newcommand{\Tbb}{\mathbb{T}}

\DeclareMathOperator{\id}{Id}

\DeclareMathOperator{\SD}{SD}

\DeclareMathOperator{\im}{im}
\DeclareMathOperator{\diam}{diam}

\DeclareMathOperator{\curl}{curl}
\DeclareMathOperator{\dist}{dist}
\DeclareMathOperator{\rank}{rank}
\DeclareMathOperator{\trace}{trace}
\DeclareMathOperator{\spn}{span}

\DeclareMathOperator{\supp}{supp}

\newcommand{\ii}{\mathrm{i}}
\newcommand{\Id}{\mathrm{Id}}
\newcommand{\set}[2]{\left\{\, #1 \ \ \textup{\textbf{:}}\ \ #2 \,\right\}}
\newcommand{\setn}[2]{\{\, #1 \ \ \textup{\textbf{:}}\ \ #2 \,\}}
\newcommand{\setb}[2]{\bigl\{\, #1 \ \ \textup{\textbf{:}}\ \ #2 \,\bigr\}}

\newcommand{\setBB}[2]{\biggl\{\, #1 \ \ \textup{\textbf{:}}\ \ #2 \,\biggr\}}

\newcommand{\norm}[1]{\|#1\|}

\newcommand{\abs}[1]{|#1|}

\newcommand{\dpr}[1]{\langle #1 \rangle}

\newcommand{\dprb}[1]{\bigl\langle #1 \bigr\rangle}

\newcommand{\dprBB}[1]{\biggl\langle #1 \biggr\rangle}

\newcommand{\ddprb}[1]{\bigl\langle\hspace{-2.5pt}\bigl\langle #1 \bigr\rangle\hspace{-2.5pt}\bigr\rangle}

\newcommand{\cl}[1]{\overline{#1}}
\newcommand{\di}{\mathrm{d}}
\newcommand{\dd}{\;\mathrm{d}}

\newcommand{\N}{\mathbb{N}}
\newcommand{\R}{\mathbb{R}}
\newcommand{\C}{\mathbb{C}}

\newcommand{\loc}{\mathrm{loc}}

\newcommand{\toweak}{\rightharpoonup}
\newcommand{\toweakstar}{\overset{*}\rightharpoonup}

\newcommand{\embed}{\hookrightarrow}

\newcommand{\sbullet}{\begin{picture}(1,1)(-0.5,-2)\circle*{2}\end{picture}}
\newcommand{\frarg}{\,\sbullet\,}
\newcommand{\BV}{\mathrm{BV}}
\newcommand{\BD}{\mathrm{BD}}

\newcommand{\BVY}{\mathbf{BVY}}

\DeclareMathOperator{\Tan}{Tan}

\DeclareMathOperator{\proj}{pr}

\newcommand{\term}[1]{\emph{#1}}

\newcommand{\proofstep}[1]{\textit{#1}}

\makeatletter
\DeclareFontFamily{OMX}{MnSymbolE}{}
\DeclareSymbolFont{MnLargeSymbols}{OMX}{MnSymbolE}{m}{n}
\SetSymbolFont{MnLargeSymbols}{bold}{OMX}{MnSymbolE}{b}{n}
\DeclareFontShape{OMX}{MnSymbolE}{m}{n}{
	<-6>  MnSymbolE5
	<6-7>  MnSymbolE6
	<7-8>  MnSymbolE7
	<8-9>  MnSymbolE8
	<9-10> MnSymbolE9
	<10-12> MnSymbolE10
	<12->   MnSymbolE12
}{}
\DeclareFontShape{OMX}{MnSymbolE}{b}{n}{
	<-6>  MnSymbolE-Bold5
	<6-7>  MnSymbolE-Bold6
	<7-8>  MnSymbolE-Bold7
	<8-9>  MnSymbolE-Bold8
	<9-10> MnSymbolE-Bold9
	<10-12> MnSymbolE-Bold10
	<12->   MnSymbolE-Bold12
}{}

\let\llangle\@undefined
\let\rrangle\@undefined
\DeclareMathDelimiter{\llangle}{\mathopen}%
{MnLargeSymbols}{'164}{MnLargeSymbols}{'164}
\DeclareMathDelimiter{\rrangle}{\mathclose}%
{MnLargeSymbols}{'171}{MnLargeSymbols}{'171}
\makeatother

\def\XXint#1#2#3{{\setbox0=\hbox{$#1{#2#3}{\int}$} 
		\vcenter{\hbox{$#2#3$}}\kern-.5\wd0}}

\newcommand{\eps}{\varepsilon}
\renewcommand{\phi}{\varphi}
\renewcommand{\hat}{\widehat}

\newcommand{\mres}{\mathbin{\vrule height 1.6ex depth 0pt width
		0.13ex\vrule height 0.13ex depth 0pt width 1.3ex}}

\title[Higher integrability for measures satisfying a PDE constraint]{Higher integrability for \\ measures satisfying a PDE constraint}

\author[Arroyo-Rabasa]{Adolfo Arroyo-Rabasa}
\address{\textit{A. Arroyo-Rabasa:} Mathematics Institute, University of Warwick, Coventry CV4 7AL, UK.}
\email{adolforabasa@gmail.com}

\author[De Philippis]{Guido De Philippis}
\address{\textit{G.~De Philippis:} Courant Institute of Mathematical Sciences, New York University, 251 Mercer St., New York, NY 10012, USA.}
\email{guido@cims.nyu.edu}

\author[Hirsch]{Jonas Hirsch}
\address{\textit{J.~Hirsch:}  Mathematisches Institut, Universit\"at Leipzig, Augustus Platz 10, D04109 Leipzig, Germany}
\email{hirsch@math.uni-leipzig.de}

\author[Rindler]{Filip Rindler}
\address{\textit{F.~Rindler:} Mathematics Institute, University of Warwick, Coventry CV4 7AL, UK.}
\email{F.Rindler@warwick.ac.uk}

\author[Skorobogatova]{Anna Skorobogatova}
\address{\textit{A.~Skorobogatova:} Department of Mathematics, Fine Hall, Princeton University, Washington Road, Princeton NJ 08540, USA.}
\email{as110@princeton.edu}

\begin{document}
	
	\begin{abstract} We establish higher integrability estimates for constant-coefficient systems of   linear PDEs 
		\[
		\Acal \mu = \sigma,
		\]
		where $\mu \in \Mcal(\Omega;V)$ and $\sigma\in \Mcal(\Omega;W)$ are vector measures and the polar $\frac{\di \mu}{\di |\mu|}$ is uniformly close to a convex cone of $V$ intersecting the wave cone of $\Acal$ only at the origin. More precisely, we prove local compensated compactness estimates of the form
		\[
		\|\mu\|_{\Lrm^p(\Omega')} \lesssim |\mu|(\Omega) + |\sigma|(\Omega), \qquad \Omega' \Subset \Omega.
		\]
		Here, the exponent $p$ belongs to the (optimal) range $1 \leq p < d/(d-k)$, $d$ is the dimension of $\Omega$, and $k$ is the order of $\Acal$. We also obtain the limiting case $p = d/(d-k)$ for canceling constant-rank operators. We consider applications to compensated compactness and {applications to the theory of} functions of bounded variation and bounded deformation.
		\vspace{4pt}
		
		\noindent\textsc{Keywords:} $\mathcal{A}$-free measure, PDE constraint, BV, BD, compensated compactness.
		
		\vspace{4pt}
		
		\noindent\textsc{Date:} \today{}.
	\end{abstract}

	\maketitle

	\section{Introduction}
	Let $V,W$ be finite-dimensional inner product vector spaces. We consider a constant-coefficient homogeneous linear differential operator $\Acal : \Dcal'(\R^d;V) \to \Dcal'(\R^d;W)$ of order $k$, which acts on smooth maps $u$ as
	\begin{equation}\label{eq:A}
		\Acal u = \sum_{|\alpha| = k} A_\alpha\partial^\alpha u, \qquad A_\alpha \in \mathrm{Lin}(V,W),
	\end{equation}
	where $\alpha =(\alpha_1,\dots,\alpha_d) \in \Nbb_0^d$ is a multi-index with modulus $|\alpha| := \alpha_1 + \dots + \alpha_d$, and $\partial^\alpha$ stands for the composition of the distributional partial derivatives $\partial_1^{\alpha_1} \cdots \partial_d^{\alpha_d}$. As usual, the \emph{principal symbol} of $\Acal$ is the $k$-homogeneous polynomial map 
	\[
	\Abb(\xi) := (2\pi \mathrm{i})^k\sum_{|\alpha| = k}  A_\alpha\xi^\alpha \in \mathrm{Lin}(V,W), \qquad \xi \in \R^d,
	\]
	where $\xi^\alpha := \xi_1^{\alpha_1} \cdots \xi_d^{\alpha_d}$. Given an open set $\Omega$ of $\R^d$, we shall consider  $V$-valued measures $\mu \in \Mcal(\Omega;V)$ satisfying the distributional PDE constraint
	\begin{equation}\label{eqn:Amu_pde}
		\Acal \mu = \sigma \qquad\text{in $\Dcal'(\Omega;W)$,}
	\end{equation}
	where $\sigma \in \Mcal(\Omega;W)$. In the special case that $\sigma \equiv 0$, we say that $\mu$ is an \emph{$\Acal$-free measure}.
	
	An object of pivotal importance is the \emph{wave cone} associated to $\Acal$, which is defined as
	\[
	\Lambda_\Acal := \bigcup_{|\xi| = 1} \ker \Abb(\xi) \subset V. 
	\]
	Observe that $\Lambda_\Acal$ characterizes those amplitudes on which the operator fails to be elliptic, i.e., 
	\begin{equation}\label{eqn:A_elliptic}
		v \in V \setminus  \Lambda_\Acal \quad\iff\quad \text{$|\Abb(\xi)v| \geq c|\xi|^k$ for some $c > 0$ and all $\xi \in \R^d$.}
	\end{equation}

	The wave cone $\Lambda_\Acal$ also plays a critical role in the theory of \term{compensated compactness}~(see, e.g.,~\cite{Tartar79,Tartar83,DiPerna85,Murat78,Murat79,Murat81}). Classically, the theory of compensated compactness for elliptic systems tells us that if $L$ is a linear subspace of $V$ with no $\Lambda_\Acal$-connections, namely (cf.~\cite[Section~2.6]{Muller99} or~\cite[Section~8.8]{Rindler18book})
	\[
	v - v' \notin \Lambda_\Acal \qquad \text{for all distinct $v,v' \in L$,}
	\]  and if $(u_j)$ is a sequence of $\Acal$-free functions satisfying
	\begin{align*}
			u_j & \toweak u \quad \text{in $\Lrm^p$},\\
			\dist(u_j,L) &\to 0 \quad \text{in measure,}
		\end{align*} 
		then $u \in \Crm^\infty(\Omega;L)$ and
		\[
		u_j \to u \quad \text{in measure}.
		\]
		Naturally, the choice to work with $\Lrm^p$ for $p > 1$ rules out any ($\Lrm^1$-)concentrations along the sequence.
		If $p = 1$, the same result holds if instead of the weak $\Lrm^p$-convergence, one assumes that $u_j\Lcal^d$ converges to $u \Lcal^d$ weakly* in the sense of measures. This, in turn, may be understood as an $\Lrm^1_\mathrm{w}$ (weak-$\Lrm^1$) compensated compactness result, which generally is not well-suited to rule out mass concentrations. The guiding question behind the present work is to investigate in what form compensated compactness theory can be extended to an $\Lrm^1$-context that prevents mass concentration. Assume we are given an $\Lrm^1$-bounded sequence of $\Acal$-free maps that take values in a cone that only intersects the wave cone at the origin and that converges weakly* in $\Mcal(\Omega;V)$ to some limit. Then, is it possible to show that the measures are actually $\Lrm^q$-maps for some $q > 1$ and that the weak* convergence can be improved to a stronger notion of convergence?  A first guess is to consider the  restriction 
		\begin{equation}\label{eq:perturbation}
			\dist\left(\frac{u(x)}{|u(x)|},L\right) \le \eps,
		\end{equation}
		for a sufficiently small $\eps> 0$. As we shall see later, this restriction is robust enough  to  establish \emph{a priori} $\Lrm^p$ higher-integrability estimates of the form
		\[
		u \in \Lrm^p(\Omega) \quad \Longrightarrow \quad  \|u\|_{\Lrm^p(\Omega')} \lesssim \|\Acal u\|_{L^1(\Omega)} + \|u\|_{L^1(\Omega)}
		\]
		for some $p > 1$ depending only on $d$ and $\Acal$ (see Remark~\ref{rem:a_priori}). However, the constraint~\eqref{eq:perturbation} being \emph{double-sided} (non-convex) allows for \emph{convex integration} methods that provide counterexamples to $\Lrm^p$ integrability. A classic example of this is contained in the work of Astala et al.~\cite{Astala08}, where it is shown that the double-sided constraint allows for the construction of almost $k$-quasiconformal curl-free fields with infinite $L^{1 + \delta}$-norm. More precisely, they proved (cf. Theorem~\ref{thm:astala} and Proposition~\ref{prop:astala} in the last section) that, for every $\delta > 0$, there exists an integrable curl-free field $u_\delta : \Omega \subset \R^2 \to \R^{2 \times 2}$ satisfying  
		$$
		\dist\left(\frac{u_\delta(x)}{|u_\delta(x)|},L_0\right) = C\delta \quad \text{for a.e. $x \in \Omega$, \quad but \quad $u_\delta \notin L^{1 + \delta}_\loc(\Omega)$}
		$$ 
		where $L_0 \subset \R^{2 \times 2}$ is the subspace of conformal matrices (a subspace with no rank-one connections). Hence, in order to establish higher-integrability estimates (without assuming a priori $\Lrm^p$ regularity), we shall henceforth work under a convexity (or {one-sided}) assumption 
		\[
		u(x) \in \mathcal \Kcal \quad \text{at a.e. $x \in \Omega$},
		\]
		where $\Kcal$ is a convex subset of $\set{y \in V}{\dist(y,L) \le \eps |y|}$, the $\eps$-cone about $L$. Notice that this convexity constraint is essential to guarantee the existence of $\Lrm^p$-regularizations by means of standard mollification.

		%
		%

	\subsection*{Notation}In all that follows $\Omega \subset \R^d$ is an open (not necessarily bounded) subset of $\R^d$. To state  the higher-integrability estimates, we work with Lipschitz  (open, bounded, and connected) domains $\Omega' \Subset \Omega$. We write
		\[
		S_V \coloneqq \set{v \in V}{v \cdot v = 1},
		\]
		to denote the unit sphere in $V$.
	
		Our main result establishes compensated $\Lrm^p$-regularity for functions with $\Lrm^p$-bounded $\Acal$-gradients when the polar of the measure is ($\Lrm^\infty$-close) uniformly close to a convex subset of $L$ with no $\Lambda_\Acal$-connections. In this regard, our results are reminiscent of the Gagliardo--Nirenberg estimates for gradients and the classical estimates for injective elliptic systems:

	\begin{theorem}[Higher integrability]\label{thm:Lp}
		Let $\Acal$ 
		be a homogenenous linear PDE operator of order $k$, from $V$ to $W$. Let $L$ be a subspace of $V$ satisfying the ellipticity condition
		\begin{equation}\label{eq:connections}
			L \cap \Lambda_\Acal = \{0_V\}
		\end{equation}
		and let
		\[
		\begin{aligned}
			p &\in \Big[1,\frac{d}{d-k} \Big)  && \text{if $k < d$},\\
			p &\in [1,\infty)  && \text{if $k \geq d$}.
		\end{aligned}
		\]
		Then, for every compactly contained subdomain $\Omega' \Subset \Omega$  there exists a positive constant $\eps = \eps(p,d,\Acal,L,\Omega')$ with the following property: If $\mu \in \Mcal(\Omega;V)$ satisfies
		\[
		\Acal \mu = \sigma \qquad \text{in $\Dcal'(\Omega;W)$}
		\] 
		for some $\sigma \in \Mcal(\Omega;W)$, and 
		\begin{equation*}
			\frac{\di \mu}{\di |\mu|}(x) \in \Ccal\qquad \text{for $|\mu|$-almost every $x \in \Omega$}
		\end{equation*}
		where $\Ccal \subset V$ is a convex set satisfying
		\begin{equation*}
			\dist\left(\Kcal \cap S_V,L\right) \le  \eps,
		\end{equation*}
		then $\mu \in \Lrm^p(\Omega';V)$ and
		\begin{equation}\label{e:Lp-est}
			\|\mu\|_{\Lrm^p(\Omega')} \le C \big( |\mu|(\Omega)+ |\sigma|(\Omega)\big).
		\end{equation}
	for some constant $C$ depending solely on  and $d,\Acal, L,\Omega'$ and $\dist(\Omega',\partial \Omega)$.
	\end{theorem}
	
	Similar to classical elliptic regularity, functions that solve the homogeneous PDE problem satisfy optimal integrability estimates (even in the case when the order of the operator is smaller than the space dimension):
		
	
	\begin{corollary}[Higher integrability for $\Acal$-free measures]\label{rem:Afree}
		If $\Acal$ and $\mu \in \Mcal(\Omega;V)$ are as in the previous theorem and
		\[
		\Acal \mu = 0 \qquad \text{in $\Dcal'(\Omega;W)$,}
		\] 
		then the estimates above hold in the range 
		$p \in [1,\infty)$ (for possibly smaller constants $\eps$ and $C$ with the same parameter dependencies), even in the case $k < d$.
	\end{corollary}

	To illustrate this result, we can record the following example, which follows immediately from the fact that the wave cone for the divergence operator consists of all singular matrices (see also Propositions~\ref{prop:grad}-\ref{prop:sgrad} on related results for the gradient and symmetric gradient operators).
	
	\begin{example}\label{e:div} Let $d \ge 2$.
		If  $\mu \in \Mcal(\Omega;\R^{d \times d})$ has divergence-free rows, i.e.,
		\[
		\mathrm{Div} (\mu) \coloneqq \biggl(\sum_{j = 1}^d \partial_j \mu_j^i \biggr)_i = 0, \qquad  i = 1,\dots,d,
		\]
		and its polar $\frac{\di \mu}{\di |\mu|}$ is sufficiently close to $\{I_n\}$ where $I_n$ is the identity matrix, then $\mu$ is locally $\Lrm^p$-integrable for every $1 < p< \infty$. Notice that under this constraint, one can think of the divergence of $\mu$ as a perturbation of the gradient since $\mathrm{Div} ( I_n \rho ) = D\rho$ for all $\rho : \R^n \to \R$.
	\end{example}

	%

	In general, one can not expect Theorem \ref{thm:Lp} to hold for the limiting exponent $p = d/(d-k)$. In order to establish a limiting estimate for $p = d/(d-k)$ (when $k < d$), we need to require further that $\Acal$ is a canceling operator as defined by Van~Schaftingen~\cite[Definition~1.3]{VanSchaftingen13} and that its symbol satisfies the constant-rank property. Then we obtain the following:

	\begin{theorem}[Limiting estimate]\label{thm:limiting}Let $\Acal, L$ be as in the previous theorem and  assume additionally that $\Acal$ is of \emph{constant rank}, i.e.,
		\[
		\rank \Abb(\xi) = \mathrm{const}  \qquad \text{for all $\xi \in \R^d \setminus \{0\}$,}
		\]
		and \emph{canceling}, i.e.,
		\[
		\bigcap_{|\xi| = 1} \im \Abb(\xi) = \{0_W\}. 
		\]
		If $k < d$, then for every compactly contained subdomains $\Omega' \Subset \Omega$  there exists a positive constant $\eps = \eps(d,\Acal,L,\Omega')$ with the following property: If $\mu \in \Mcal(\Omega;V)$ satisfies
		\[
		\Acal \mu = \sigma \qquad \text{in $\Dcal'(\Omega;W)$}
		\]
		for some $\sigma \in \Mcal(\Omega;W)$, and 
		\begin{equation*}
			\frac{\di \mu}{\di |\mu|}(x) \in \Ccal \qquad \text{for $|\mu|$-almost every $x \in \Omega$,}
		\end{equation*}
		where $\Ccal \subset V$ is a convex set such that  
		\begin{equation*}
			\dist\left(\Kcal \cap S_V,L\right) \le  \eps,
		\end{equation*}
		then $\mu  \in \Lrm^{{d}/{(d-k)}}(\Omega';V)$ and 
		\[
		\|\mu\|_{\Lrm^{\frac{d}{d-k}}(\Omega')} \le C \big( |\mu|(\Omega)+ |\sigma|(\Omega)\big).
		\]
		for some constant $C$ depending solely on $d,\Acal,L,\Omega'$ and $\dist(\Omega',\partial \Omega)$.
		%
	\end{theorem}

	\begin{remark}[Results for inhomogeneous PDEs]
		It is worth mentioning that in the above theorems, there is no obstruction to taking $\Acal$ to be an \emph{inhomogeneous} PDE operator instead, other than the fact that the range of $p$ for which the estimates in Theorem~\ref{thm:Lp} and Theorem~\ref{thm:limiting} would be reduced to $\big[1, \frac{d}{d-k'}\big)$ and $\big[1,\frac{\ell}{\ell-k'}\big)$ respectively, where $$k' :=  k - \max \setn{l < k}{\Acal \ \text{has a derivative of order} \ l}.$$
		This is due to the mapping properties of the multiplier operators associated to the lower-order terms that would arise in the proofs.
	\end{remark}

	Theorem~\ref{thm:Lp} implies the following \emph{compensated compactness} result:
	
	\begin{corollary} \label{cor:CC}
		Let $(\mu_j) \subset \Mcal(\Omega;V)$ satisfy
		\[
		\Acal\mu_j = \sigma_j \quad\text{in $\Dcal'(\Omega;W)$}
		\]
		and
		\[
		\sup_{j \in \Nbb} \, \bigl\{\abs{\mu_j}(\Omega) + \abs{\sigma_j}(\Omega) \bigr\}< \infty.
		\] 
		Let $p$, $\Omega'$ and  $\eps$ be as in Theorem~\ref{thm:Lp}. Further assume that  \begin{equation*}
			\frac{\di \mu}{\di |\mu|}(x) \in \Ccal \qquad \text{for $|\mu|$-almost every $x \in \Omega$,}
		\end{equation*}
		for some convex set $\Ccal \subset V$ such that  
		\begin{equation*}
			\dist\left(\Kcal \cap S_V,L\right) \le  \eps,
		\end{equation*}
		Then, $(\mu_j) \subset \Lrm^p(\Omega';V)$ and
		\[
		\text{$\bigl\{|\mu_j|^q\bigr\}_j$ is equiintegrable on $\Omega'$ for all $1  \le q < p$.}
		\]
		Moreover, 
		\[
		\mu_j \toweakstar \mu \text{ as measures} \qquad \text{implies} \qquad \mu_j \toweak \mu \text{ in $\Lrm^q(\Omega';V)$},
		\]
		and
		\[
		\left\{ \begin{aligned}
			&\mu_j \to \mu \text{ in measure (as maps), or}\\
			&\mu_j \to \mu \text{ a.e.\ (as maps)}
		\end{aligned} \right\}
		\qquad\text{implies}\qquad
		\mu_j \to \mu \text{ in $\Lrm^q(\Omega';V)$}.
		\]
	\end{corollary}
	
	We refer to Section~\ref{sc:counterexamples} for some counterexamples on related statements one might conjecture but which are false. This also sheds some light on the differences to the classical $\Lrm^p$-case. In particular, Example~\ref{ex:laminates} conveys that one cannot expect (local) strong $\Lrm^p$-compactness in the previous corollary.

	\begin{remark}
		Our results are also related to Question~1 in~\cite{BateOrponen20} and to~\cite{DeRosaSerreTione19?}. In this regard, recall that an \term{Alberti representation} of $\mu \in \Mcal^+(\R^d)$ is a (non-zero) measure
		\[
		\nu =  \int_0^1 \Hcal^1\mres E_t \dd t \qquad \text{for some family of 1-rectifiable sets} \ \{E_t\}_{t \in I},
		\]
		such that $\mu \ll \nu$. We say that $\mu$ has $d$ \term{independent} Alberti representations $\{\nu^1,...,\nu^d\}$ with corresponding families $\{E_t^i\}_{t \in I}$, $i=1,...,d$, if
		\[
		\spn \bigl\{\Tan(E_t^1,x), \ldots, \Tan(E_t^d,x)\bigr\} = \R^d
		\]
		for $(\Hcal^1\mres E_t)$-almost every $x$ and almost every $t \in [0,1]$,
		where $\Tan(E,x)$ is the tangent space of $E$ at $x$; see~\cite[Section~2]{AlbertiMarchese16} for a more detailed introduction. If $\mu$ has $d$ independent Alberti representations, then one may construct (see~\cite[Corollary~6.5]{AlbertiMarchese16} or~\cite[Lemma~4.2.5]{DePhilippisMarcheseRindler17}) $d$ normal $1$-currents $T_i$ such that $\mu \ll \|T_i\|$ for all $i = 1,\ldots,d$,  and
		\[
		\spn \bigl\{ \vec{T}_1(x), \ldots, \vec{T}_d(x) \} = \R^d \qquad \text{for $|\mu|^s$-almost every $x \in \R^d$.}
		\]
		Hence,~\cite[Corollary~1.12]{DePhilippisRindler16} tells us that $\mu \ll \Lcal^d$.
		Our results can be interpreted as more quantitative information about the density $\frac{\di\mu}{\di\Lcal^d}$. In particular, under the assumptions that \(\vec{T}_i\) is sufficiently close to \(e_i\), we can deduce that $\frac{\di\mu}{\di\Lcal^d} \in \Lrm^p$ for all $1 < p < \infty$, compare with Example \ref{e:div}. The case  $p = 2$ and  $d=2$ is also considered in~\cite{BateOrponen20}.
	\end{remark}

	We remark that all our present results are all \emph{perturbative} in the sense that they apply only to a small conical neighborhood of the linear space $L$. It is an interesting question to ask whether the results can be extended to the case of \emph{any} closed and convex cone that intersects the wave cone only at the origin. In an earlier version of this article, we posed the following (false) conjecture:
	
	\begin{conjecture}\label{conj:main}
		Suppose $\mu \in \Mcal(\Omega;V)$ satisfies~\eqref{eqn:Amu_pde} and its polar satisfies
		\[
		\frac{\di\mu}{\di|\mu|}(x) \in \Ccal \qquad \text{for} \ \text{$|\mu|$-a.e.} \ x \in \Omega,
		\]
		where $\Ccal \subset V$ is 
		a convex cone satisfying the ellipticity assumption
		\[
		\Ccal \cap \Lambda_\Acal = \{0_V\}.
		\]
		Then, analogous $\Lrm^p$ estimates to those in Theorem \ref{thm:Lp} hold for $\mu$. 
	\end{conjecture}

	However, the examples provided in \cite{GRS22} disprove the validity of Conjecture \ref{conj:main} for the full range of exponents suggested in Theorem~\ref{thm:Lp}. There, the authors pose a slight modification of our conjecture by suggesting the higher integrability holds when the range of exponents is further restricted (see~\cite[Conjecture 1.1]{GRS22}):
	\[
		p \in [1,q_{\max}), \qquad q_{\max} = q_{\max}(\mathbb A,\Kcal) \le p.
	\]
which indeed appears reasonable.
	
	We remark that a positive resolution of such a conjecture would imply the following compensated compactness result, for $p$ below $q_{\max}$ (as defined above): For a uniformly bounded sequence of measures $\mu_j \in \Mcal(\Omega;V)$ ($\sup_j |\mu_j|(\Omega) < \infty$) satisfying $\Acal \mu_j = \sigma_j \in \Mcal(\Omega;W)$ with $\sup_j |\sigma_j|(\Omega) < \infty$ and polars $\frac{\di \mu_j}{\di |\mu_j|}$ lying in a (fixed) cone $\Ccal$ as above, the family $\{\mu_j\}$ is ($\Lrm^p$-)equiintegrable, meaning that the $\mu_j$ are maps (not measures) and there are no ($\Lrm^p$-)concentrations in the sequence.
	
	To motivate the expectation of improved integrability discussed above, we will show the following rigidity result:
	
	\begin{proposition}\label{prop:rigid} Let $\Ccal \subset V$ be a convex cone satisfying 
		\[   
		\Ccal \cap \Lambda_\Acal = \{0\}.
		\]
		Then, there are no non-zero solutions $u \in \Lrm^1(\R^d;V)$ of the homogeneous system
		\[ \left\{
		\begin{aligned}
			\, \Acal u &= 0, \\
			u &\in \Ccal \ \text{a.e.} 
		\end{aligned}
		\right.
		\]
	\end{proposition}
	\begin{remark}\label{rem:rigid}
		This rigidity result also holds under the weaker assumption that $\Acal$ is co-canceling on $\Ccal$, i.e., 
		\[
		\Ccal \cap 	\bigcap_{|\xi| = 1} \ker \Abb(\xi) = \{0\}.
		\]
	\end{remark}

	\subsection*{Structure of the paper}
	We conclude this introduction by briefly describing the structure of the paper. In Section \ref{sec:elliptic} we present some applications of our main results. In Section \ref{sc:prelim} we record some useful fact about elliptic operators while in Section \ref{sec:rank} we  prove a localized version of Van Schaftingen's estimate for canceling operators which is instrumental to prove Theorem \ref{thm:limiting}. In Section \ref{sec:main} we prove our main results and in  Section \ref{sc:counterexamples} we present some counterexamples showing that the convexity of the nonlinear constraint cannot be relaxed. 

	%
	%
	%
	\subsection*{Acknowledgments} 
	
	AAR, FR and AS have received funding from the European Research Council (ERC) under the European Union's Horizon 2020 research and innovation programe, grant agreement No 757254 (SINGULARITY).  GDP is partially supported by the NSF Grant DMS-2055686. JH was partially supported by the German Science Foundation DFG in context of the Priority Program SPP 2026 ``Geometry at Infinity''.

	\section{Compensated embeddings for elliptic operators}\label{sec:elliptic}
	
	In this section we give an account of applications of our results  to deduce higher integrability results for elliptic operators and other well-known PDE constraints.

	\subsection{Higher integrability for elliptic operators}The same argument can be generalized to all elliptic operators  as follows. Let us fix a homogeneous elliptic operator $\Bcal$ of order $k$ from a finite-dimensional normed vector space $U$ to another such space  $V$, i.e., for a smooth map $u \colon \R^d \to U$,
	\begin{equation}\label{eq:BB}
		\Bcal u := \sum_{|\alpha| = k} B_\alpha \partial^\alpha u, \qquad B_\alpha \in \mathrm{Lin}(U,V).
	\end{equation}
	We make the \emph{ellipticity} assumption that there is a constant $c > 0$ such that the symbol $\Bbb(\xi)$ of $\Bcal$ satisfies
	\[
	|\Bbb(\xi)v| \ge c |\xi|^k|v| \qquad \text{for all $\xi \in \R^d$ and $v \in U$.}
	\]  
	A parade of examples of constant-coefficient elliptic operators can be found in~\cite{VanSchaftingen13} (see also~\cite{ArroyoRabasa20?} where also other less known examples are discussed)

	Elliptic operators possess (see below) an annihilator operator $\Acal$, that is, a constant-coefficient linear PDE operator from $V$ to $V$ such that
	\[
		\im \Bbb(\xi) = \ker \Abb(\xi) \quad \text{for all $\xi \in \R^d \setminus \{0\}$}
	\]
	and
	\[
	\Acal(\Bcal u) = 0 \qquad \text{in $\Dcal'(\Omega;W)$}
	\]
	for all $u \in \Mcal(\Omega;U)$.
	
	If the polar of $\Bcal u$ takes values away from $\Lambda_\Acal$ in the sense of Theorem~\ref{thm:Lp} and Remark~\ref{rem:Afree}, then we can make use of the improved $\Lrm^p$-estimates for $\Acal$-free measures. To make this more precise, let us introduce the following suitable annihilator. By the ellipticity, $\Bbb(\xi)^* \Bbb(\xi) \colon U \to U$ (where here and in the following we have identified the primal and dual spaces) defines a linear isomorphism  for every $\xi \in \R^d \setminus \{0\}$. Consider the homogeneous operator $\Acal$ from $V$ to $V$ of order $2kd$ associated to the symbol
	\[
	\Abb(\xi) := \det[\Bbb(\xi)^* \Bbb(\xi)] \id_V - \Bbb(\xi) \circ \mathrm{adj}[\Bbb(\xi)^* \Bbb(\xi)]\circ \Bbb(\xi)^*.
	\]
	Here, the adjugate matrix $\mathrm{adj}(Q)$ of $Q \in U \otimes U$ is  the transpose of the cofactor matrix of $Q$. In particular, this ensures that $\mathrm{adj}[\Bbb(\xi)^* \Bbb(\xi)]$ is a matrix-valued polynomial of order $2k(d-1)$ and thus $\Abb(\xi)$ is indeed the symbol of a $2kd$-homogeneous operator $\Acal$ with constant coefficients. Moreover, the adjugate matrix of an invertible $Q \in U \otimes U$ satisfies the identity $\mathrm{adj}(Q)\circ Q = \det(Q)\id_U$, which implies that 
	\[
	\im \Bbb(\xi) = \ker \Abb(\xi), \qquad  \xi \in \R^d \setminus \{0\}.
	\]
	In particular, the wave cone associated to $\Acal$ is given as
	\begin{equation}\label{eq:dualcone}
		\Lambda_\Acal = \bigcup_{|\xi|=1} \im \Bbb(\xi) \subset V.
	\end{equation}
	Hence, by a localization argument and an application of the Fourier transform we conclude that $\Acal \circ \Bcal = 0$, as desired. Notice that, due to~\eqref{eq:dualcone}, the $L$-ellipticity condition~\eqref{eq:connections} translates to into the requirement
	\begin{equation}\label{eq:B}
		L \cap \im \Bbb(\xi) = \{0\} \qquad \text{for all $\xi \in \R^d$}.
	\end{equation}
	Therefore, we may apply the higher-regularity results for $L$-elliptic systems of $\Acal$-free measures to deduce higher-integrability for $\Bcal u$ provided that the polar of $\Bcal u$ takes values near $L$. The precise statement is the following:
	
	\begin{theorem}\label{thm:elliptic} Let $\Bcal$ be a homogeneous elliptic operator from $U$ to $V$, of order $k$, and let $L$ be a subspace of $V$ with no $\Bbb$-connections, i.e.,
		\[
		L \cap \im \Bbb(\xi) = \{0\} \qquad \text{for all $\xi \in \R^d$}. 
		\] 
		Then, for every $1 < p < \infty$ and every compactly contained subdomain $\Omega' \Subset \Omega$  there exist positive constant  $\eps = \eps(p,d,\Acal,L,\Omega')$  with the following property: if $u \in \Lrm^1(\Omega;U)$ is  such that $\Bcal u \in \Mcal(\Omega;V)$ and 
		and 
		\begin{equation*}
			\frac{\di \Bcal u}{\di |\Bcal u|}(x) \in \Ccal \qquad \text{for $|\mathcal Bu|$-almost every $x \in \Omega$,}
		\end{equation*}
		for some convex set $\Ccal \subset V$ such that  
		\begin{equation*}
			\dist\left(\Kcal \cap S_V,L\right) \le  \eps,
		\end{equation*}
		then $u\in  \Wrm^{k,p}(\Omega';\R^d)$ and
		\[
		\|u\|_{\Wrm^{k,p}(\Omega')} \le C \,\big( \|u\|_{\Lrm^1(\Omega)} + |\Bcal u|(\Omega)\big) \,, \qquad  \Omega' \Subset \Omega,
		\]
		for some constant $C$ depending on $p,d,L,\Bcal,\Omega'$ and $\dist(\Omega',\partial \Omega)$.
	\end{theorem}
	\begin{proof}
		If $L = \{0\}$, then the choice $\eps < 1$ implies that $\Bcal u \equiv 0$ and hence the result follows from the Riesz representation theorem (or the Lax--Milgram theorem). Let us therefore assume that $L$ is non-trivial and let $\Acal$ be an annihilator of $\Bcal$ as above, so that
		\[
		\im \Bbb(\xi) = \ker \Abb(\xi) \qquad \text{for all $\xi \in \R^d \setminus \{0\}$.}
		\]
		Notice that $\Acal$ is also not the trivial operator since by assumption 
		\[
		L \cap \ker \Abb(\xi) = \{0\} \qquad \text{for all $\xi \in \R^d$}.
		\]
		Fix $1 < p < \infty$ and let $\eps$ be the smallness constant in Theorem~\ref{thm:Lp} (which by Corollary~\ref{rem:Afree} exists for $p$ in this range). It then follows by the aforementioned theorem and its corollary that $\Bcal u \in \Lrm^p(\Omega')$ and
		\[
		\|\Bcal u\|_{\Lrm^p(\Omega')} \le C|\Bcal u|(\Omega), \qquad \Omega' \Subset \Omega.
		\]
		The sought Sobolev estimates for $u$ then follow from a further localization argument and the classical Calder\'on--Zygmund estimates for elliptic operators (see Theorem~\ref{thm:mihlin}). This finishes the proof.
	\end{proof}

	\subsection{Higher integrability for functions of bounded variation}
	
	For an open domain $\Omega \subset \R^d$, we define the space of functions of bounded variation~\cite{AmbrosioFuscoPallara00book} as
	\[
	\BV(\Omega) := \setb{u \in \Lrm^1(\Omega;\R^m)}{Du \in \Mcal(\Omega;\R^m \otimes \R^d)}.
	\]
	We say that a subspace $L$ of $\R^m \otimes \R^d$ has no rank-one connections if 
	\[
	L \cap \setb{a\otimes \xi}{\xi \in \R^d, \; a \in \R^m} = \{0\}.
	\]
	We write
	\[
	S_{m,d} \coloneqq \set{A \in \R^m \otimes \R^d}{|A|^2 = \trace(A^TA) = 1},
	\]
	to denote the set of tensors of norm one in $\R^m \otimes \R^d$.
	Applying Theorem~\ref{thm:elliptic} with $\Bcal := D$ (or Theorem~\ref{thm:Lp} with $\Acal := \curl$) yields the following regularity result:
	
	\begin{proposition}	\label{prop:grad}
		Let $L$  be a space of matrices in $\R^m\otimes \R^d$ with no rank-one connections, let $1 < p < \infty$ and let $\Omega' \Subset \Omega$. Then, there exists a constant $\eps = \eps(p,d,D,L,\Omega')$  such that if $u \in \BV(\Omega;\R^m)$ satisfies
		and 
		\begin{equation*}
			\frac{Du}{|Du|}(x) \in \Ccal \qquad \text{for $|Du|$-almost every $x \in \Omega$,}
		\end{equation*}
		for some convex cone $\Ccal$ of the space of matrices $\R^m \otimes \R^d$ satisfying  
		\begin{equation*}
			\dist\Big(\Ccal \cap S_{m,d},L\Big) \le \eps,
		\end{equation*}
		then $u \in \Wrm^{1,p}(\Omega';\R^m)$, 
		and 
		\[
		\|u\|_{\Wrm^{1,p}(\Omega')} \le C \, \|u\|_{\BV(\Omega)}
		\]
		for some $C = C(p,d,D,L,\Omega',\dist(\Omega',\partial \Omega))$.
	\end{proposition}
	
	\begin{remark}
		A similar statement holds for the analogous bounded variations spaces associated to the $k$'th-order Hessian
		\[
		D^k u := (\partial^\alpha u)_{|\alpha| = k},
		\]
		which defines an elliptic operator from $\R^m$ to $\R^m \otimes E_k(\R^d)$, where $E_k(\R^d)$ denotes the symmetric tensors of order $k$ on $\R^m$. Then,
		\[\BV^k(\Omega;\R^m) = \setb{u \in \Lrm^1(\Omega;\R^d)}{D^k u \in \Mcal(\Omega;\R^m \otimes E_k(\R^d))},
		\]
		and the $L$-ellipticity condition takes the form
		\[
		L \cap \setb{a \otimes \underbrace{\xi \otimes \cdots \otimes  \xi}_{\text{$k$-times}}}{\xi \in \R^d, \; a \in \R^m} = \{0\}.
		\]
	\end{remark}
	
	\subsection{Higher integrability of symmetric gradients}
	
	The space $\BD(\Omega)$ of functions of bounded deformation~\cite{TemamStrang80,AmbrosioCosciaDalMaso97} is
	\[
	\BD(\Omega) := \setb{u \in \Lrm^1(\Omega;\R^d)}{Du + Du^T \in \Mcal(\Omega;E_2(\R^d)}.
	\]
	We have the following rigidity result for functions of bounded deformation: 
	
	\begin{proposition}\label{prop:sgrad}Let $L$ be a subspace of $E_2(\R^d)$ satisfying
		\[
		L \cap \setn{a \otimes b + b \otimes a}{a,b \in \R^d} = \{0\},
		\]
		 let $1 < p < \infty$, and let $\Omega' \Subset \Omega$. There exists a constant $\eps = \eps(p,d,E,L,\Omega')$ 
		such that if $u \in \BD(\Omega)$ satisfies 
		\[
		\frac{\mathrm{d}Eu}{\mathrm{d}|Eu|}(x) \in \Kcal \qquad \text{for $|Eu|$-almost every $x \in \Omega$,}
		\]
		for some convex cone $\mathcal K$ of $E_2(\R^2)$ such that
		\[
		\dist\Big(\mathcal K \cap S_{d,d},L \Big) \le \eps,
		\]
		then $u \in \Wrm^{1,p}(\Omega';\R^d)$ 
		and 
		\[
		\|u\|_{\Wrm^{1,p}(\Omega')} \le C \, \|u\|_{\BD(\Omega)},
		\]
		for some $C = C(p,d,L,E,\Omega',\dist(\Omega',\partial \Omega))$.
	\end{proposition}

	%

	\section{Preliminaries}\label{sc:prelim}
	In this section, we recall some well-known facts of Sobolev spaces and Calder\'on--Zygmund estimates that will be useful in the following sections.
	%
	%
	Given a function \(u\) in the Schwartz class \(\mathcal S(\R^d)\), we define its Fourier transform by
	\[
	\mathcal F(u)(\xi)=\hat u(\xi)=\int u(x)e^{-2\pi\ii x \cdot \xi}dx.
	\] 
	We recall  the classical   H\"ormander--Mihlin multiplier theorem (\cite[Theorem 5.2.7]{Grafakos14book1}).
	
	\begin{theorem}[Mihlin]\label{thm:mihlin} 
		Assume that $m \in \Crm^\infty(\R^d \setminus \{0\};\mathrm{Lin}(V,W))$ is smooth and it satisfies
		\[
		|\partial ^{\alpha} m(\xi)|\le C_{\alpha, d} |\xi|^{-\alpha}
		\]
		for all multi-indexes \(\alpha\).
		Then the multiplier
		\[
		T[f] :=  \mathcal F^{-1}\left[m \widehat u\right], \qquad u \in \Scal(\R^d;V),
		\]
		can be extended to a continuous linear functional from $\Lrm^p(\R^d;V)$ to $\Lrm^p(\R^d;W)$, i.e., 
		\[
		\|T[f]\|_{\Lrm^p} \le C \|f\|_{\Lrm^p}
		\]
	\end{theorem}
	Given a domain \(\Omega\subset \R^d\) we define the Sobolev space \(W^{k,p}(\Omega)\) as the class of functions \(u\) for which 
	\[
	\|u\|_{ \Wrm^{k,p}(\Omega)} := \left(\sum_{|\alpha| \le  k} \int_\Omega |\partial^\alpha u|^p \right)^\frac{1}{p}<+\infty
	\]
	where \(\partial^\alpha u\) is the \(\alpha\) distributional derivative. Recall also that if \(\Omega=\R^d\), then
	\begin{equation}\label{e:sobolev}
		\|u\|_{ \Wrm^{k,p}(\R^d)}\approx \|\mathcal F^{-1}((1+|\xi|^{k})\widehat u\|_{ \Lrm^p(\R^d)}
	\end{equation}
	We also denote by \(W^{-k,p'}(\Omega)\) the dual of \(W^{k,p}_0(\Omega)\), where \(p'=p/(p-1)\).
	
	We recall the following embedding  results for Sobolev spaces, see for instance~\cite{AdamsHedberg96book} or~\cite{Leoni09book}:
	
	\begin{lemma}\label{lem:embed} Let $k$ be an integer, $\ell$ a non-negative integer and  $1 < p < \infty$ and let \(\Omega\) be either \(\R^d\) or a Lipschitz domain. Then,
		\begin{enumerate}[(a)]\setlength{\itemsep}{1pt}
			\item $\Wrm^{k + \ell,p}(\Omega) \embed \Wrm^{k,\frac{dp}{d - \ell p}}(\Omega)$, when $0 < \ell p < d$;
			\item $\Wrm^{k+\ell,p}(\Omega) \embed \Wrm^{k,q}_\loc(\Omega)$ for all $p \le q < \infty$, when $\ell p \ge d$;
			\item $\Wrm^{k+\ell,p}(\Omega) \embed \Crm^{k}(\Omega)$, when $\ell p > d$ and $k \ge  0$.
			\item In particular, $\Mcal(\Omega) \embed \Wrm^{-k,q}(\Omega)$ for all $1 < q < \frac d{d-k}$ and $0 < k < d$ or all $q > 1$ if $k \geq d$.
		\end{enumerate}
	\end{lemma}
	
	It will be convenient to introduce the following exponent.
	
	\begin{definition} \label{def:qell}
		Let $1 < q < \infty$ and $\ell \in \N$. Define
		\begin{equation}\label{eq:notation}
			q(\ell) := \begin{cases}
				\frac{dq}{d -\ell q} & \ell q < d,\\
				q(\ell-1) &  \ell q \ge d,
			\end{cases}
		\end{equation}
		and
		\[
		q(0) := q.
		\]
	\end{definition}
	
	Notice that in particular $q \le q(\ell) < \infty$ and that
	\[
	[q(r)](\ell) = q(r+\ell).
	\]
	From Lemma~\ref{lem:embed} we have
	\begin{equation} \label{eq:qell_embed}
		\Wrm^{-k+\ell,q}(\R^d) \embed \Wrm^{-k,q(\ell)}_\loc(\R^d).
	\end{equation}

	\section{Generalized Laplace operators}We prove standard existence and regularity estimates for generalized Laplace--Beltrami operators associated with elliptic operators. For an elliptic homogeneous operator $\Bcal$ of order $k$ from $U$ to $V$ as in~\eqref{eq:BB}, we define the homogeneous operator $\triangle_\Bcal \coloneqq \Bcal^*\Bcal$ of order $2k$ from $U$ to $U$ that is associated to the symbol $\Bbb(\xi)^* \Bbb(\xi)$. Here, $\Bcal^*$ denotes the formal adjoint of $\Bcal$ and $\Bbb(\xi)^*$ denotes the Hermitian adjoint of $\Bbb(\xi)$. 
	
	\begin{remark}For every elliptic operator $\Bcal$, the operator $\triangle_\Bcal$ is positive-definite in the sense that its symbol $\Bbb(\xi)^*\Bbb(\xi)$ satisfies 
		\begin{align*}
			\Bbb(\xi)\Bbb(\xi)^*a \cdot a = |\Bbb(\xi)a|^2 \ge C |\xi|^{2k}|a|^2 \quad \text{for all $a \in U$.}
		\end{align*}
	\end{remark}
	\begin{remark}
		Notice that $\triangle_D = - \Delta$, where $D = (\partial_1,\dots,\partial_d)$ is the gradient operator and $\Delta = \sum_{j = 1} \partial_{jj}$ is the classical Laplacian operator.
	\end{remark}

	\subsection{Existence and regularity of generalized Laplacians}
	
	%
	
	The following existence for the Dirichlet problem
	is a direct consequence of the standard $\Lrm^p$-theory for strongly elliptic systems:

	\begin{lemma}\label{lem:std} Let $\Bcal$ be an elliptic operator from $U$ to $V$ of order $k$. Define the generalized Laplace--Beltrami operator  $\triangle_\Bcal$ as the operator associated to the symbol $\Bbb(\xi)^*  \Bbb(\xi)$.
		Then, for $1 < p < \infty$, the distributional equation
		\[
		(\Id+\triangle_\Bcal) u = f, \qquad f \in \Lrm^p(\R^d;U),
		\]
		has a unique solution $u \in   \Wrm^{2k,p}(\R^d;U)$ satisfying,
		\[
		\| u\|_{W^{2k,p}} \le C \|f\|_{\Lrm^p} 
		\]
		with $C = C(p,d,\Bcal)$.
	\end{lemma}

	\begin{proof} 
		
		Since $\triangle_\Bcal$ is positive-definite, the map
		\[
		\Tbb(\xi) := (\id+\Bbb(\xi)\Bbb(\xi)^{*})^{-1}
		\]
		is well-defined and smooth on $\R^d $.  Given a Schwartz  function  \(f\in \mathcal S(\R^d;U) \), we set 
		\[
		u\coloneqq\mathcal F^{-1}( \Tbb(\xi) \widehat f).
		\]
		It is immediate to verify that \(u\) belongs to \(\mathcal S(\R^d;U)\) and that it solves the equation. 
		Moreover, by the Mihlin multiplier Theorem~\ref{thm:mihlin} and \eqref{e:sobolev},
		\[
		\begin{split}
			\| u\|_{W^{2k,p}(\R^d)}&\le C\|\mathcal F^{-1}((1+|\xi|^{2k})\widehat u)\|_{ \Lrm^p(\R^d;U)}
			\\
			&=\|\Fcal^{-1}( (1+|\xi|^{2k}) \Tbb(\xi) \widehat f)\|_{ \Lrm^p(\R^d;U)}\\
			& \le C \|f\|_{\Lrm^p(\R^d;U)}.
		\end{split}
		\]
		A simple approximation argument concludes the existence part of the proof. To prove uniqueness just note that the positivity of $\triangle_\Bcal$ implies that any Schwartz distributions solving 
		\[
		(\Id+\triangle_\Bcal) u = 0,
		\]
		also satisfies \(\hat u(\xi)=0\) for \(\xi \in \R^d\).
	\end{proof}

	For later use we now prove a perturbative version of the previous invertibility result. This will allow us to prove the main $\Lrm^p$  estimates for elliptic systems by a duality argument.
	

	\begin{corollary}\label{cor:perturbation} Let $\Omega \subset \R^d$ be a Lipschitz domain and let $\Bcal$ as be as in the previous lemma. Consider the operator $\id+\triangle_\Bcal - \Rcal$, where $\Rcal$ is an operator from $U$ to $U$ of order $2k$ with measurable coefficients of the form 
		\[
		\Rcal = \sum_{|\alpha| = 2k} R_\alpha(x) \partial^\alpha, \qquad R_\alpha \in \mathrm{Lin}(U,U).
		\]
		For $1  < p < \infty$, there exists a constant $\delta = \delta(d,p,\Omega,\Bcal) > 0$ such that, if
		\[
		\sup_{|\alpha| = k} \norm{R_\alpha}_{\Lrm^\infty} \le \delta,	
		\]
		then the distributional equation
		\[
		(\Id+\triangle_\Bcal - \Rcal)u = f, \qquad f \in \Lrm^p(\Omega;U),
		\]
		has a solution $u \in \Wrm^{2k,p}(\Omega;U)$ satisfying
		\[
		\|u\|_{W^{2k,p}(\Omega;U)} \le C \|f\|_{\Lrm^p(\Omega;U)} 
		\]
		for some constant $C = C(p,d,\Bcal)$.
	\end{corollary}
	\begin{proof} For a given function $\eta \in \Lrm^p(\R^d;U)$, we write $(\id+\triangle_\Bcal)^{-1} \eta \in \Wrm^{2k,p}(\R^d)$ to denote the solution of $(\id +\triangle_\Bcal )w = \eta$ constructed in the previous lemma, for which we have
		\[
		\|(\id +\triangle_\Bcal)^{-1} \eta\|_{\Wrm^{2k,p}} \le C(p,d,\Bcal) \|\eta\|_{\Lrm^p}.
		\]
		Define the linear map 
		\[
		T[u] := (\id+\triangle_\Bcal)^{-1}\Rcal[u]
		\]
		for $\Wrm^{2k,p}(\R^d;U)$. The second statement of the previous lemma implies that 
		\[
		\|T\|_{\Wrm^{2k,p} \to \Wrm^{2k,p}} \le C\delta,
		\]
		where the constant $C$ depends on $p,d$ and $\Bcal$.
		Now, provided that $\delta$ is sufficiently small (i.e., strictly smaller than  $1/C$), the von-Neumann series bounds imply that $(\id - T)$ is invertible with $\|(\id - T)^{-1}\|_{\Wrm^{2k,p} \to \Wrm^{2k,p}} \le 2$. In particular, we may define 
		\[
		u := (\id - T)^{-1}(\id+\triangle_\Bcal)^{-1}{\tilde f} \in  \Wrm^{2k,p}(\R^d;U),
		\]
		where $\tilde f$ is the trivial extension by zero of the function $f$. By construction, $u$ satisfies the distributional equation
		\[
		(\id +\triangle_\Bcal - \Rcal)u = (\id +\triangle_\Bcal )[(\id - T)u] = {\tilde f} \quad \text{on $\R^d$},
		\]
		and the global bound 
		\[
		\|u\|_{\Wrm^{2k,p}(\R^d;U)} \le C(p,d,\Bcal) \|f\|_{\Lrm^p(\Omega;U)}.
		\]
		{The assertion follows by taking the restriction of $u$ to $\Omega$.}
		%
	\end{proof}

	\section{Canceling operators}\label{sec:rank}

	The aim of this section is to show a local version for constant-rank operators of Van~Schaftingen's canceling estimates for elliptic operators~\cite[Theorem~1.3]{VanSchaftingen13}. This will be employed in the proof of the borderline estimate of Theorem~\ref{thm:limiting}.
	
	More precisely, we show that if $\Acal$ is a canceling operator of constant rank and $\mu \in \Mcal(\Omega;V)$, then 
	\[
	\Acal \mu = \sigma \in \Mcal(\Omega) \quad \Longrightarrow \quad \sigma \in \Wrm^{-1,d/(d-1)}(\Omega'), \qquad \Omega' \Subset \Omega,
	\]
	and  the $\Wrm^{-1,d/(d-1)}(\Omega')$-norm of $\sigma$ can be quantitatively estimated in terms of its total variation measure $|\sigma|$ up to a constant depending solely on $d,\Abb$ and $\dist(\Omega',\Omega)$.
	
	\subsection{Representation of operators of constant rank} In order to establish the local cocanceling estimates we need to recall the kernel representation for operators of constant rank, along with some basic properties of their kernel representation.  To this end, let us fix a linear PDE operator $\Acal$ of order $k$ from $V$ to $W$ satisfying
	\begin{equation}\label{eq:cr}
		\rank \Abb(\xi) = \mathrm{const}  \qquad \text{for all $\xi \in \R^d \setminus \{0\}$,}
	\end{equation}
	and denote by $\pi(\xi): V \to V$ the orthogonal projection from $V$ to $(\ker \Abb(\xi))^\perp$. A classical result of Schulenberger and Wilcox~\cite{SchulenbergerWilcox72} states that if~\eqref{eq:cr} is verified, then the map $\xi \mapsto \pi(\xi)$ is analytic on $\R^d \setminus \{0\}$. Since $\pi(\xi)$ is also $0$-homogeneous, the map $\xi \mapsto \pi(\xi)$ defines an $\Lrm^p$-Fourier multiplier. Murat~\cite{Murat81} used this to define an ``$\Abb$-representative'', which for any $u \in \Ecal'(\R^d;V)$, can be defined as
	\[
	u_\Abb := \Fcal^{-1}(\pi \widehat u).
	\]
	Notice that, by construction, we have $\Abb \circ \pi = \Abb \circ [\id - \proj_{\ker \Abb(\xi)}] = \Abb$ (with $\proj_{\ker \Abb(\xi)}$ the orthogonal projection onto $\ker \Abb(\xi)$) and therefore
	\begin{equation}\label{eq:represent}
		\Abb \widehat {u_\Abb} = \Abb \widehat u, \qquad \Acal u_\Abb = \Acal u.
	\end{equation}
	It is well-known that~\eqref{eq:cr} implies that $\xi \mapsto \Abb(\xi)^\dagger$ belongs to $\Crm^\infty(\R^d;\mathrm{Lin}(W,V))$, where $M^\dagger$ denotes the Moore-Penrose inverse of $M$ (see, for instance~\cite[Propositon 8]{Arroyo_Simental}).
	On the other hand, using that $\Abb^\dagger$ is homogeneous of degree $-k$, one can show that $\Abb^\dagger$ extends to a tempered distribution on $\R^d$ and, denoting its extension also by $\Abb^\dagger$, the kernel $K_\Abb := (2\pi\mathrm{i})^k \Fcal^{-1}\Abb^\dagger$ is locally integrable on $\R^d \setminus \{0\}$ and belongs to $\Crm^\infty(\R^d \setminus \{0\};\mathrm{Lin}(W,V))$. The precise statement, which follows from a minor modification of~\cite[Lemma~2.1]{BousquetVanSchaftingen14}, is the following:
	\begin{lemma}\label{lem:kernel}
		Let $\Acal$ be a linear PDE operator of order $k$, from $V$ to $W$. If $\Acal$ satisfies the constant-rank property, then $K_\Abb := (2\pi\mathrm{i})^k \Fcal^{-1}\Abb^\dagger$ is locally integrable and satisfies
		\[K_\Abb \in \Crm^\infty(\R^d \setminus \{0\};\mathrm{Lin}(W,V)).
		\]
		In particular, $K_\Abb$ is a fundamental solution of $\Acal$ in the sense that
		\[
		u_\Abb = K_\Abb \star \Acal u  \qquad \text{for all $u \in \Crm^\infty_c(\R^d;V)$}. 
		\]
		Moreover, for every $\ell > k - d$, the map $D^\ell K_\Abb$ is homogeneous of degree $-d + k - \ell$.	
	\end{lemma} 
	
	Recalling the seminal ideas of Fonseca and M\"uller~\cite{FonsecaMuller99}, we can use this representation of $u_\Abb$ to deduce Sobolev estimates as one would do for elliptic operators. One simple verifies that that the map
	\[
	\xi \mapsto m(\xi) = \xi^\alpha|\xi|^{k-|\alpha|}\Abb(\xi)^\dagger
	\] 
	is homogeneous of degree $0$ and smooth on the sphere. By Mihlin's theorem, we deduce that the linear map
	\[
	T[f] \mapsto \Fcal^{-1}\bigl[m \widehat {f}\,\bigr], \qquad f \in \Crm_c^\infty(\R^d;V),
	\]
	is bounded from $\Lrm^p$ to $\Lrm^p$. In turn, the identity 
	\[
	\widehat{\partial^\alpha u_\Abb}(\xi) = \xi^\alpha\Abb(\xi)^\dagger \widehat{\Acal u}(\xi) = m(\xi) \, \Fcal(I_{k - |\alpha|} \star \Acal u)(\xi),
	\]
	where $I_s = c_{s,d}\,|\frarg|^{-n+s}$ is the  Riesz potential of order $s$,
	conveys that (see~\cite[Sec.~1.2]{AdamsHedberg96book})
	\begin{align*}
		\|D^{j} u_\Abb\|_{\Lrm^p} & \le C \|I_{k-j} \star \Acal u\|_{\Lrm^p}  \approx C\|\Acal u\|_{\dot\Wrm^{-{(k-j)},p}},  \label{eq:bound}
	\end{align*}
	holds for every positive integer $j \in (k-d,k)$ for some $C = C(d,p,k)$. Here, $\dot\Wrm^{-s,p}(\R^d)$ is the dual of $W^{s,p'}(\R^d)$, where  the homogeneous Sobolev space $\dot\Wrm^{s,p}(\R^d)$ is  defined as the closure of $\Wrm^{s,p}(\R^d)$ with respect to the semi-norm $\|D^su\|_{\Lrm^p}$.
	
	Notice that when $k > d$, one cannot directly infer the $\Lrm^p$-estimate 
	\begin{equation}\label{eq:no}
		\|u_\Abb\|_{\Lrm^p} \le C_p \|\Acal u\|_{\dot\Wrm^{-k,p}},
	\end{equation}
	which is a key estimate when dealing with localization arguments.
	The issue is that for $k$ larger than the dimension $d$, one can no longer make use of the equivalence between the Riesz potential norm $\|I_k f\|_{\Lrm^p}$ and the negative homogeneous norm $\|f\|_{\dot \Wrm^{-k,p}}$ induced by the duality $\Wrm^{-k,p} = (\Wrm^{k,p'})^*$. Instead, this $\Lrm^p$ estimate will be addressed by exploiting the kernel representation from Lemma~\ref{lem:kernel}.

	Finally, we recall that constant-rank operators are precisely the class of operators for which an exact annihilator exists. This was proven by Raita~\cite{Raita19a} (see also~\cite{Arroyo_Simental}), who showed that~\eqref{eq:cr} is equivalent with the existence of an operator $\Lcal$ from $W$ to $W$ (of the form~\eqref{eq:A}) with symbol $\Lbb$ satisfying
	\begin{equation}\label{eq:exact}
		\im \Abb(\xi) = \ker \Lbb(\xi) \qquad \text{for all $\xi \in \R^d \setminus \{0\}$.}
	\end{equation}

	\subsection{Local estimates}Now that we are equipped with a kernel representation, we show that 
	if $\phi \in \Crm^\infty_c(\R^d;[0,1])$ is a cut-off function satisfying $\phi \equiv 1$ on an open neighborhood of $\cl \omega \subset \R^d$ (where $\omega$ is open and bounded), then
	\[
	\Acal u \in \Mcal(\R^d;W) \quad \Longrightarrow \quad  (\phi u)_\Abb \in \Wrm^{k-1,q}(\omega) \;\, \text{for all $1 \le q < d/(d-1)$}.
	\]
	The proof of this result is an adaptation for constant-rank operators of the one for elliptic operators given in~\cite[Lemma 2.2]{Raita19b}, which itself is based on an argument of H\"ormander (see~\cite[Theorem~4.5.8]{Hormander94book}). Since the adaptation of the proof is not entirely trivial, we include it here for the convenience of the reader.
	\begin{lemma}\label{lem:regularity}
		Let $\Acal$ be a linear PDE operator of order $k$ from $V$ to $W$ that satisfies the constant-rank property. Let $\Omega' \subset \Omega$ be an open and bounded subdomain and let $\phi \in \Dcal(\R^d;[0,1])$ be a test function satisfying 
		\[
		\mathbf 1_{\Omega'} \le \phi \le \mathbf 1_\Omega \quad \text{and} \quad 0 < 
		\dist(\supp \partial^\alpha\phi,\Omega') \le \dist(\Omega',\partial\Omega)
		\] 
		for all $|\alpha| \ge 1$. If $1 \le q < d/(d-1)$ and $\mu \in \Mcal(\R^d;V)$ satisfies
		\[
		\Acal \mu = \sigma \in \Mcal(\R^d;W),
		\]
		then
		\[
		\|D^{k-1}(\phi \mu)_\Abb\|_{\Lrm^q(\Omega')}  \le C \, \big(|\mu|(\Omega) + |\Acal u|(\Omega)\big)
		\]
		for some $C = C(q,d,\Abb,\dist(\Omega',\supp D\phi),\dist(\Omega',\partial\Omega),|\supp \phi|)$.
	\end{lemma}
	\begin{proof}
		Define $w := \phi \mu \in \Dcal'(\R^d;V)$. By the discussion above and the representation of constant-rank operators we get
		\[
		w_\Abb = K_{\Abb}\star \Acal (\phi \mu),
		\]
		where $K_{\Abb}$ is the kernel given by Lemma~\ref{lem:kernel}. Following the argument in~\cite[Lemma~2.2]{Raita19b}, we find that
		\begin{align*}
			D^{k-1} w_\Abb & = (D^{k-1} K_{\Abb}) \star \Acal(\phi \mu) \\
			& = (D^{k-1} K_{\Abb}) \star (\phi\sigma) + (D^{k-1} K_{\Abb}) \star [\Acal,\phi](\mu)\\
			& =: \mathrm{I} + \mathrm{II}.
		\end{align*}
		Where the commutator $[\Acal,\phi]$ is the $(k-1)$\textsuperscript{th} order operator defined as 
		\begin{equation}\label{e:com}
			[\Acal,\phi] (\eta) =\Acal(\phi \eta) - \phi  \Acal(\eta).
		\end{equation}
		First, we bound $\|\mathrm{I}\|_{\Lrm^q}$. We distinguish two cases: if $d > 1$, then we may apply the result of Lemma~\ref{lem:kernel} with $\ell = k -1 > k - d$ to deduce that $D^{k-1} K_{\Abb}$ is homogeneous of degree $-d + 1$ and smooth away from $0 \in \R^d$.  Then, $|D^{k-1} K_{\Abb}| \lesssim |\frarg|^{-d+1} \in \Lrm^q_\loc(\R^d)$. Then, letting $R := 2\diam(\Omega)$, we see from Young's convolution inequality that
		\begin{align*}
			\|\mathrm{I}\|_{\Lrm^q(\Omega')} & \lesssim \|I_1 \star (\phi \sigma)\|_{\Lrm^q(\Omega')} \\
			& \le \|I_1\|_{\Lrm^q(B_R)} \cdot |\phi \sigma|(\R^d),
		\end{align*}
		where $I_1(x) = c_n |x|^{-d+1}$ is the $1$-Riesz potential. 
		If $d = 1$, then from the discussion in~\cite[Chapter~3]{Raita19b} it follows that $|D^{k-1} K_{\Abb}| \lesssim 1 + \log |\frarg|$ and therefore $D^{k-1} K_{\Abb} \in \Lrm^p_\loc(\R^d)$ for all $1 \le p < \infty$. In this case Young's inequality yields
		\[
		\|\mathrm{I}\|_{\Lrm^q(\Omega')} \lesssim \|D^{k-1} K_{\Abb}\|_{\Lrm^q(\Omega')} \cdot |\sigma|(\Omega).
		\]
		In both cases the constants involved into the inequalities depend solely on $q,d$ and $\Acal$.

		Next we bound $\|\mathrm{II}\|_{\Lrm^q}$. A modification of the proof~\cite[Lemma~2.2]{Raita19b} shows that $\mathrm{II}$ is smooth on $\Omega'$. Here, we need to obtain more information on the actual bounds of $\mathrm{II}$, for which we shall first assume that $\mu \in \Lrm^1_\loc(\Omega)$. 
		By the Leibniz rule, $[\Acal,\phi]$ is a (inhomogeneous) linear PDE operator of 
		order (at most) $k-1$ such that
		\[
		\supp\,[\Acal,\phi] \subset \supp {\partial^\gamma\phi}
		\]
		for all $\gamma$ with $1 \leq |\gamma| \leq k$.

		
		Let $x \in \Omega'$. By assumption $0 < \delta := \inf_{|\gamma| \ge 1} \dist(\supp \partial^\gamma\phi,\Omega') \le |y - x| \le \dist(\Omega',\partial \Omega)$ for all $y \in \supp D\phi$. Let $S$ be the closed annulus $\cl{B_\lambda} \setminus B_\delta$, where $\lambda := \dist(\Omega',\partial\Omega)$. Then, it follows from integration by parts that
		\begin{align*}
			|(D^{k-1} K_{\Abb} \star [\Acal,\phi](u))(x)| & \le \|K_\Abb\|_{\Wrm^{2k-1,\infty}(R)} \cdot \|\phi\|_{\Wrm^{2k,\infty}(R)} \cdot \|\mu\|_{\Lrm^1(S)} \\
			& \le C |\mu|(\Omega)
		\end{align*}
		with a constant $C = C(q,d,\Acal,\lambda,\delta)$. Here, we used the regularity of the kernel $K_\Abb \in \Crm^\infty(\R^d\setminus \{0\};\mathrm{Lin}(W,V))$ in passing to the last inequality. This proves that
		\[
		\|\mathrm{II}\|_{\Lrm^q} \le C |\supp \phi|^\frac{1}{q} |\mu|(\Omega),
		\]
		which together with the bound for $\mathrm{I}$ implies the desired bound when $\mu$ is in $\Lrm^1_\loc(\Omega)$. The general case for $\mu \in \Mcal(\Omega;V)$ follows from the bounds for $\mathrm{I}$ and $\mathrm{II}$ above and a standard mollification argument.
	\end{proof}
	
	We are now ready to prove the main result of this section, namely the following local canceling estimates.
	
	\begin{theorem}
		\label{thm:localVS}
		Let $\Acal$ be a linear PDE operator of order $k$, from $V$ to $W$. Further assume that $\Acal$ is canceling and satisfies the constant-rank property. If $\mu \in \Mcal(\Omega;V)$ satisfies
		\[
		\Acal \mu = \sigma \qquad \text{in $\Dcal'(\Omega;W)$}
		\]
		for some $\sigma \in \Mcal(\Omega;W)$, then $\sigma \in \Wrm^{-1,\frac{d}{d-1}}_\loc(\Omega)$. Moreover, for all compact subdomains $\Omega' \Subset \Omega$, 
		\[
		\left|\int_{\Omega'} \phi \cdot \mathrm d\sigma \right| \le C \big(|\mu|(\Omega) +  
		|\sigma|(\Omega) \big) \cdot \|D\phi\|_{\Lrm^d(\Omega')} \qquad \text{for all $\phi \in \Crm_c^\infty(\Omega';W)$},
		\]
		where $C = C(d,\Abb,\dist(\Omega',\partial\Omega),\diam(\Omega'))$.
	\end{theorem}
	\begin{proof}
		Let $\phi_1,\phi_2 \in \Crm_c^\infty(\Omega;[0,1])$ be two test functions satisfying $\mathbf 1_{\Omega'} \le \phi_1 \le \phi_2$ and whose supports are nested sufficiently far from each other so that $(\partial^\alpha \phi_2) \phi_1 = 0$ for all $|\alpha| > 0$ and also
		\[
		\|\partial^{\alpha}\phi_i\|_\infty \le C(k)\dist(\Omega',\partial \Omega)^{-2k}, \qquad |\alpha| \le 2k.
		\] 
		Define $w := \phi_2 \mu$ and let $q := (2d -1)/2(d-1)$. By the estimates of Lemma~\ref{lem:regularity} and the properties of $\phi_1,\phi_2$ we deduce that
		\[
		\|w_\Abb\|_{\Wrm^{k-1,q}(\Omega')} \le C\big( |\mu|(\Omega) + |\sigma|(\Omega)\big),
		\]
		for some $C = C(d,\Abb,\dist(\Omega',\partial \Omega),\diam(\Omega'))$. Next, we define $v := \phi_1 w_\Abb$, which satisfies
		\begin{align*}
			\Acal v & = \phi_1 \Acal(\phi_2 \mu) + [\Acal,\phi_1](w_\Abb) \\
			& = \phi_1 \phi_2 \sigma + \phi_1[\Acal,\phi_2](\mu) + [\Acal,\phi_1](w_\Abb)
		\end{align*}
		in the sense of distributions on $\Omega$. Now, use that the derivatives of $\phi_2$ vanish on $\{\phi_1 \equiv 1\}$ to find that
		\[
		\Acal v = \phi_1 \sigma + [\Acal,\phi_1](w_\Abb)\qquad \text{in $\Dcal'(\Omega;W)$},
		\]
		where the right-hand side is the sum of a measure and an $\Lrm^q$-function, both of which can be estimated in terms of $|\mu|(\Omega) + |\sigma|(\Omega)$ up to a constant depending on $d,\Abb, \diam(\Omega')$ and $\dist(\Omega',\partial\Omega)$.
		Identifying $v$ with its trivial extension by zero on $\R^d$, we have constructed a function $v \in \Wrm^{k-1,q}(\R^d;V)$ satisfying
		\begin{gather}
			\Acal v \in \Mcal(\R^d;W), \\ 
			|\Acal v|(\R^d) \le C\big( |\mu|(\Omega) + |\sigma|(\Omega)\big),\label{eq:t2}\\
			\Acal v = \sigma \quad \text{in the sense of distributions on $\Omega'$}\label{eq:t1}.
		\end{gather}
		We are now in position to apply global cocanceling estimates from~\cite[Theorem~1.4]{VanSchaftingen13}: Let $\Lcal$ be an annihilator of $\Acal$ such that~\eqref{eq:exact} holds. Then, the canceling assumption on $\Acal$ implies that $\Lcal$ is cocanceling as defined in~\cite[Definition~1.3]{VanSchaftingen13}. Applying a version of the aforementioned theorem with (with $\Lcal$ and $f = \Acal v$) to conclude that
		\[
		\sup_{\substack{\phi \in \Crm_c^\infty(\Omega'),\\\|D\phi\|_{L^d} \le 1}}	\left|\int_{\Omega'} \varphi \cdot\mathrm{d}\sigma \right|   \stackrel{\eqref{eq:t1}}\le \|\Acal v\|_{\dot \Wrm^{-1,d/(d-1)}} \lesssim  |\Acal v|(\R^d)  \stackrel{\eqref{eq:t2}}\lesssim |\mu|(\Omega) + |\sigma|(\Omega),
		\]
		where the constants carried through the inequalities depend only on $d,\Acal,\Omega'$ and $\dist(\Omega',\partial\Omega)$. 
	\end{proof}

	\section{Proofs of the main results}\label{sec:main}
	We are now in position to give a proof of our main theorem.
	
	\begin{proof}[Proof of Theorem~\ref{thm:Lp}]Without loss of generality we may assume that $V = \R^m$, $L = \R^\ell \times \{0\}^{m - \ell}$ and $\eps \le 2^{-1}$.
		Write $\mu = P\nu$ with $\nu \in \Mcal^+(\Omega)$ and $P \in \Lrm^\infty(\nu;\R^m)$ to denote the polar decomposition of $\mu$.  For a vector $z \in \R^m$ let us adopt the notation $z = (v,w)$ where $v \in \R^\ell$ and $w \in \R^{m - \ell}$. Recall that, under these conventions our main assumption on $P$ reads as
		\begin{equation}\label{eq:pc}
			P(x) \in \Kcal_\eps \subset \setBB{(v,w) \in \R^m}{|w| \le \frac{\eps |v|}{\sqrt{1 - \eps^2}}}, \quad  \text{$\nu$-a.e. $x \in \Omega$},
		\end{equation}
	where $\Kcal_\eps$ is convex
		and $\eps > 0$ is a sufficiently small real number, to be constrained below.

		\proofstep{Reduction to smooth maps.} We claim that establishing a priori estimates for smooth functions suffice. Indeed, if $\{\rho_t\}_{t > 0}$ is family of standard probability mollifiers at scale $t>0$, then our assumption that $\Kcal_\eps$ is a convex cone is crucial to deduce that $\mu_t \coloneqq \mu \star \rho_t$ satisfies 
			\[
			\mu_t(x) \in \Kcal_\eps \qquad \text{for every $x \in \Omega^t \coloneqq \set{x \in \Omega}{\dist(x,\partial\Omega)>t}$.}
			\]
			Then, the validity of the sought estimates for smooth functions leads to the estimate
			\[
			\|\mu_t\|_{\Lrm^p(\Omega')} \le C_t\big(|\mu_t|(\Omega^t) + |\sigma_t|(\Omega^t)\big)
			\] 
			where $C_t  = C(d,\Acal,L,\Omega',\dist(\Omega',\partial\Omega_t))$. As it will become apparent from the development of the proof, it holds $C_t \ge C_s$ for all $t \ge s > 0$ and therefore from Young's inequality for convolutions it holds
			\begin{align*}
				\liminf_{t \to 0^+}\|\mu_t\|_{\Lrm^p(\Omega')} & \le \liminf_{t \to 0^+}\, C_t\big(|\mu_t|(\Omega^t) + |\sigma_t|(\Omega^t)\big) \\
				& \le C\left(|\mu|(\Omega) + |\sigma|(\Omega)\right),
			\end{align*}
			where $C = C(d,\Acal,L,\Omega',\dist(\Omega',\partial\Omega)$. This proves that $\mu_t$ is uniformly bounded in $\Lrm^p(\Omega')$ provided and since $\mu_t \to \mu$ in the sense of distributions on $\Omega$, we conclude from the lower semicontinuity of the $\Lrm^p$-norm with respect to distributional convergence that
			\[
			\|\mu\|_{\Lrm^p(\Omega')}  \le C\left(|\mu|(\Omega) + |\sigma|(\Omega)\right),
			\] 
			which is precisely what we aim to prove. Hence, in all that follows, we may assume that $\mu$ is a smooth map (so that $P$ and $\nu$ are also smooth) on $\Omega$.
			
			\begin{remark}[A priori $L^p$ estimates]\label{rem:a_priori}
			This previous step is the only place where the convexity of $\Kcal$ is used (to ensure the existence of  $\Kcal$-constrained regularizations). It is worth mentioning that if the convexity assumption on $\Kcal$ is dispensed with (in the sense that it is allowed to be a double-sided cone), then our proof still gives the estimate (for $p$ as in the range of Theorem~\ref{thm:Lp}) if a priori we know that $\mu$ is $L^p$-integrable, that is,
			\[
				\mu \in L^p(\Omega) \quad \Longrightarrow \quad \|\mu\|_{L^p(\Omega')} \lesssim |\mu|(\Omega) + |\Acal \mu|(\Omega).
			\] 
To back this claim up, we simply observe the following: 
\begin{enumerate}[(a)]
\item the estimates for $q(\ell)$ derived in Step~3 (below) hold for all $1 < p < \infty$, provided that $\mu \in L^p(\Omega)$, \item the estimates in Step~4 (below) hold for all $p$ as in the range of Theorem~\ref{thm:Lp}. 
\end{enumerate}
			\end{remark}
		
		\proofstep{Step~1.} Our assumption on $L$ implies that $L \cap \Lambda_\Acal \cap \Sbb^{m-1} = \emptyset$. 
		%
		Since $L$ and $\Lambda_\Acal$ are cones,  it follows that  $\delta_L:= \dist(\Lambda_\Acal \cap \Sbb^{m-1},L) > 0$. In all of the following we  assume that $\eps \in (0,\delta_L)$, so that $\Lambda_\Acal \cap \Kcal_\eps = \{0\}$. 
		

		\proofstep{Step~2.}
		Before delving into the main argument, let us introduce some handy notation: We decompose the polar $P$ as the direct sum $P_0 + P_1$, where
		\[
		P_0 := \pi_L P, \qquad P_1 := P - \pi_L P,
		\]
		where $\pi_L \colon \R^m \to L$ is the orthogonal projection onto $L$. By construction, both $P_0$ and $P_1$ are $\Lrm^\infty$-functions and hence we may define a measurable tensor field by setting 
		\[
		M :=  \biggl(\frac{P_1 \otimes P_0}{|P_0|^2}\biggr)\chi_{\{|P| > 0\}} \in \Lrm^\infty(\Omega;\R^{m \times m}).
		%
		\]
		For future reference, we record that~\eqref{eq:pc} conveys the uniform bound
		\begin{equation}\label{eq:M}
			\|M\|_{\Lrm^\infty(\Omega)} \le \frac{|P_1|}{|P_0|} \le \frac{\eps}{\sqrt{1 - \eps^2}} \leq 2\eps.
		\end{equation}

		
		\proofstep{Step~3.} Let us first assume that $\mu$ is compactly supported in $\Omega'$ and that $\mu$ satisfies, for some $1 < q < \infty$ and some integer $0 \le \ell \le k-1$,
		\[
		\Acal \mu = \gamma \quad \text{in $\Dcal'(\Omega;V)$},
		\]
		for some distribution $\gamma \in \Wrm^{-k+\ell,q}(\Omega;W)$ supported on $\supp \mu \subset \Omega'$. Henceforth, compactly supported measures and distributions will often be considered as measures and distributions on $\R^d$ through their trivial extensions. The key argument is to express the source term $\gamma$ as a perturbation of a well-behaved elliptic operator acting on $P_0 \nu$. More precisely,  we write $\Bcal := \Acal \circ \pi_L$ so that 
		\begin{align}\label{eq:both}
			\gamma  =  \Acal \mu 
			= \Bcal (P_0 \nu) + \Acal( P_1 \nu)  
			=\big(\Bcal + \Acal  M \big) P_0 \nu.
		\end{align}
		Since $P_0\nu \in \Lrm^1(\Omega;\R^m)$, then
		\[
		\dprb{M P_0 \nu, \psi}_{\Dcal' \times \Dcal} = \dprb{P_0\nu, M^T\psi}_{\Lrm^1 \times \Lrm^\infty} \qquad \text{for all $\psi \in \Dcal(\R^d;\R^m)$}.
		\]
		In particular, the $(\Lrm^1,\Lrm^\infty)$-adjoint of
		\[
		\Rcal = \Bcal^* \Acal M
		\]
		is given by
		\[
		\Rcal^* = M^T \Acal^* \Bcal,
		\]
		which is an operator of order $2k$ that can be written in the form
		\[
		\sum_{|\alpha| = 2k} R_\alpha(x)^* \,
		\partial^\alpha
		\]
		where the coefficients $x \mapsto R_\alpha(x)^*$ are Lebesgue-measurable and satisfy 
		\begin{equation}\label{eq:eM}
			\|R_\alpha^*\|_{\Lrm^\infty} \le C \|M\|_{\Lrm^\infty} \le C_1\eps
		\end{equation}
		for some $C_1 = C_1(\Acal,L)$.
		Applying $\Bcal^*$ {and adding $P_0\nu$} to both sides of~\eqref{eq:both}, we obtain
		\begin{equation}\label{eq:transition}
			P_0 \nu+\Bcal^*\gamma = (\id+\triangle_\Bcal + \Rcal) P_0 \nu =: T \in \Dcal(\Omega;\R^m),
		\end{equation}
		where $\triangle_\Bcal := \Bcal^* \Bcal$. In particular $\supp (T) \subset \Omega'$. Integration by parts yields (recall that $\triangle_\Bcal$ is self-adjoint with respect to the $\Lrm^2$-pairing)
		\begin{equation}\label{eq:T}
				\dpr{T,\phi} = \dpr{P_0 \nu,(\id +\triangle_\Bcal + \Rcal^*)\phi}, \qquad \phi \in W^{2k,s}(\Omega;\R^m), \quad s \in (1,\infty).
			\end{equation} 
			
			In the next lines, we shall often identify smooth maps as elements of a negative Sobolev space.
		From Lemma~\ref{lem:embed} in the form~\eqref{eq:qell_embed} and in conjunction with~\eqref{eq:T}, we deduce (recall that $\Bcal^*$ is a homogeneous operator of order $k$ and $\gamma$ is a compactly supported distribution on $\Omega' \Subset \R^d$) 
		\begin{align}
			\|T\|_{\Wrm^{-2k,q(\ell)}(\Omega)} & = \|P_0 \nu+\Bcal^*\gamma\|_{\Wrm^{-2k,q(\ell)}(\Omega)} \notag\\
			&  \le \|\gamma\|_{\Wrm^{-k,q(\ell)}(\Omega)}+\|P_0\nu\|_{\Wrm^{-2k,q(\ell)}(\Omega)}
			\notag\\
			& \lesssim_{q,d,k,\Omega'} \|\gamma\|_{\Wrm^{-k+\ell,q}(\Omega)}+\|P_0\nu\|_{\Wrm^{-2k,q(\ell)}(\Omega)}.  \label{eq:s}
		\end{align}
		Here, $q(\ell)$ is the exponent from Definition~\ref{def:qell}.
		
		In all that follows we further assume $2C_1\eps \le \delta$, where $\delta = \delta(d,q(\ell)',\Omega',\Acal,L)$ is the smallness constant from Corollary~\ref{cor:perturbation} and  $q(\ell)'$ is the dual exponent of $q(\ell)$. Applying the results of the aforementioned corollary to $\id +\triangle_\Bcal + \Rcal^*$, the smallness~\eqref{eq:eM} implies that given  $f \in  \Dcal(\Omega';\R^m)$, 
		we may find $\eta \in \Wrm^{2k,q(\ell)'}(B_1(\Omega');\R^m)$ satisfying \
		\begin{equation}\label{eq:2}
			(\id+\triangle_\Bcal + \Rcal^*)\eta = f \quad \text{and} \quad \|\eta\|_{\Wrm^{2k,q(\ell)'}(B_1(\Omega'))} \le C_2 \|f\|_{\Lrm^{q(\ell)'}(\Omega')}
		\end{equation}
		where $C_2 = C_2(d,q,\Acal,L)$ and we have set 
		\[
		B_1(\Omega')=\{ x\in \R^d: \dist(x, \Omega)\le 1\}.
		\]
		Note also  that $0 \le \ell \le k-1$ so that the constant does not depend on $\ell$ but rather on the order $k$ of $\Acal$.
		Now, let $\chi \in \Crm^\infty_c(\Omega)$ be a cut-off function satisfying $\mathbf 1_{\cl{\Omega'}} \le \chi \le \mathbf 1_{\Omega \cap B_1(\Omega')}$ and 
		\[
		\|\chi\|_{C^{2k}} \le C(k)\lambda^{2k}, \qquad \lambda := 1 \wedge \dist(\Omega',\partial \Omega)^{-1}.
		\]
		Notice that with this in mind $T[\chi\eta]$ is well-defined and equals $\dpr{\pi_L[\mu],f}$. Hence, we obtain the following duality estimate:
		\begin{align}
			|\dpr{P_0\nu,f} | 
			& \le \|T\|_{\Wrm^{-2k,q(\ell)}(\Omega)} \cdot  \| \chi\eta\|_{\Wrm^{2k,q(\ell)'}(\Omega)} \notag 
			\\
			& \stackrel{\eqref{eq:s}-\eqref{eq:2}}\le C(q,d,\Acal,L,\Omega',\lambda) \,(\|\gamma\|_{\Wrm^{-k+\ell,q}(\Omega)}+\|P_0\nu\|_{W^{-2k,q(\ell)}(\Omega)}) \cdot \|f\|_{\Lrm^{q(\ell)'}(\Omega')}.  \label{eq:pear}
		\end{align} 
		By the dual definition of the $\Lrm^{q(\ell)}$-norm we conclude that 
		\begin{equation}\label{eq:ll}
			\begin{split}
				\|P_0\nu\|_{\Lrm^{q(\ell)}(\Omega')} \le C(q,d,\Acal,L,\Omega',\lambda)\,(\|\gamma\|_{\Wrm^{-k+\ell,q}(\Omega)}+\|P_0\nu\|_{W^{-2k,q(\ell)}(\Omega)}).
			\end{split}
		\end{equation}
		
		\proofstep{Step~4.} We proceed now prove the general case, which follows by a  localization argument and the previous step.
		
		Let $p$ be as in the statement of the theorem. Notice that by the definition of $p$ we may always find $q = q(p,d,k)$ satisfying
		\[
		1 < q < \frac{d}{d-1} \quad \text{and} \quad p = q(k-1) < \infty,
		\]
		where $q(k-1)$ is the exponent defined in Definition~\ref{def:qell}. Our next objective will be to establish $\Lrm^{q(k-1)}$-estimates for localizations of $\mu$. In all that follows we  assume that 
		\[
		\eps = \eps(p,d,\Acal,L,\Omega') < \min\set{ C_1^{-1}\delta(q(\ell)',d,\Acal,L,\Omega')}{0 \le \ell \le k-1},
		\]
		so that~\eqref{eq:ll} holds for all $0 \le \ell \le k-1$ and some constant $C = C(d,p,\Acal,L,\Omega',\lambda)$ whenever $\Acal \mu = \gamma \in \Wrm^{-k+\ell,q}(\Omega)$, $\supp \mu \subset \Omega'$, and $P \in \Kcal_\eps$ almost everywhere with respect to $\nu$.
		
		First, let us consider a  sequence of nested domains   $\Omega_k = \Omega' \Subset \Omega_{k-1} \cdots \Subset \Omega_{1} \Subset \Omega_{0} = \Omega$ with Lipschitz boundary  and satisfying
		\[
		\dist(\Omega_{r+1},\partial \Omega_{r}) \ge (2k)^{-1} (1 \wedge \dist(\Omega',\partial\Omega))\qquad \text{for all $r = 0,\dots,k-1$}.	
		\]
		For each such $r$, we may find a cut-off function $\phi_r \in \Crm_c^\infty(\Omega)$ satisfying
		\[
		\|D^\ell \phi_r\|_{\Lrm^\infty}\le C(k) \lambda^{k} \quad \text{for $\ell \in \{0,1,\ldots, k\}$}, \qquad \mathbf 1_{\Omega_{r+1}} \le \phi_r \le \mathbf 1_{\Omega_{r}},
		\]
		where as above $\lambda = 1 \wedge \dist(\Omega',\partial\Omega)^{-1}$.
		In all of the following, we adopt the convention that the bounding constant $C$ may increase from inequality to inequality (or from line to line), but remains to depend solely on $p,d,k,\Acal,L$ and $\Omega'$; additional dependencies such the dependency on $\lambda$ will be carried next to the constant. 
		
		Define the localized measures $\mu_r := \phi_r \mu$ for all $r = 0,\dots,k-1$. The main argument consists of bootstrapping the regularity of $\mu_r$ at each step $r = 0,\dots,k-1$, until we reach the desired $\Lrm^{q(k-1)}$-bounds for $\mu$ on $\Omega_k = \Omega'$. 
		The key is to observe that $\mu_r$ satisfies the distributional equation 
		\[
		\Acal \mu_r =\gamma_r := \phi_r \sigma + [\Acal,\phi_r](\mu_{r-1}),
		\]
		where the commutator $[\Acal,\phi]$ is defined in \eqref{e:com} and it is of order $k-1$ and whose coefficients depend solely on the principal symbol of $\Acal$ and  $\|\phi_r\|_{C^k}$.

		With regard to the regularity of the localization term $[\Acal,\phi_r](\mu_{r-1})$ (which a priori only belongs to $\Wrm^{-k,q}(\Omega)$), we will show that, at the $r$'th step of the iteration, the measure $\mu_{r}$ belongs to $\Lrm^{q(r)}$. This, in turn, will allow us to lift the regularity of $[\Acal,\phi_{r+1}](\mu_r)$ to $\Wrm^{-k,q(r+1)}$  and subsequently bootstrap it onto the regularity of $\mu_{r+1}$.
		
		Let us begin with the estimate for $r = 1$. Using that $[\Acal,\phi_r]$ is an operator of order $k-1$ with $\Crm^\infty$-coefficients depending only on $\partial^\alpha \phi_r$ ($|\alpha| \le k$) and the coefficients of $\Acal$, we have
		\begin{align*}
			\|[\Acal,\phi_1] (\mu_1)\|_{\Wrm^{-k,s}}
			& \le  C \lambda^k \|\mu_0\|_{\Wrm^{-1,s}}
		\end{align*}
		for all $1 < s < \infty$.
		Since also $\mu_{1},\gamma_1$ are compactly supported on $\Omega_1$, the bound from~\eqref{eq:ll} with $\ell = 0$  and $\Omega_1$ instead of $\Omega'$ yields
		\begin{align}
			\|P_0\nu_1\|_{\Lrm^q}
			& \le C \big(\|\phi_1\sigma\|_{\Wrm^{-k,q}} +\|P_0\nu_1\|_{\Wrm^{-2k,q}} + \|[\Acal,\phi_1] (\mu_1)\|_{\Wrm^{-k,q}}\big)  \notag\\
			& \le C \big(\|\phi_1\sigma\|_{\Wrm^{-k,q}} +\|P_0\nu_1\|_{\Wrm^{-2k,q}} + C(k) \lambda^k  \|\mu_0\|_{\Wrm^{-1,q}}\Big)  \notag \\
			&  \le  C\lambda^k \big(|\sigma|(\Omega) + |\mu|(\Omega)\big),  \label{eq:boo}
		\end{align}
		where the last inequality follows from Morrey's embedding (see Lemma~\ref{lem:embed}~(d)) and the fact that $1 < q < d/(d-1)$.
		The pointwise constraint $|\mu| \le (1 + \eps) |P_0\nu_1|$ further implies that
		\begin{align*}
			\|\mu\|_{\Lrm^q(\Omega')} & \le \|\mu_{1}\|_{\Lrm^q(\Omega_{1})} \\
			& \le (1+\eps) \|P_0\nu_1\|_{\Lrm^q} \\
			& \le (1 + \eps) C\lambda^k \big(|\sigma|(\Omega) + |\mu|(\Omega)\big).
		\end{align*}
		Notice that in the case $k=1$ (when $q(k-1) = q(0) = q \ge p$), this already proves the desired bound.

		We shall henceforth assume that $k \ge 2$. 
		Let us now address the derivation of the $\Lrm^{q(r)}$-estimates at the $r$'th step. 
		Assuming in an inductive fashion that $\mu_{r-1} \in \Lrm^{q(r-1)}$, we estimate
		\[
		\|[\Acal,\phi_{r}](\mu_r)\|_{\Wrm^{-k,q(r)}}   
		\le C \lambda^k \|\mu_{r-1}\|_{\Wrm^{-1,q(r)}}  
		\le C\lambda^k \|\mu_{r-1}\|_{\Lrm^{q(r-1)}}.
		\]
		In the last line we observed that with $q(r)$ from Definition~\ref{def:qell}, we have (see~\eqref{eq:qell_embed}) that $\Lrm^{q(r-1)}(\Omega_1) \embed \Wrm^{-1,q(r-1)(1)}(\Omega_1) = \Wrm^{-1,q(r)}(\Omega_1)$.
		
		Similarly, we recall from Lemma~\ref{lem:embed} that for $1 \le r \le k-1$,
		\[
		\|\phi_r \sigma\|_{\Wrm^{-k,q(r)}} \le C \|\phi_r \sigma\|_{\Wrm^{-k+r,q}} \le C\|\phi_r\sigma\|_{\Wrm^{-1,q}} \le C|\sigma|(\Omega).
		\]
		Using once more the bounds in~\eqref{eq:ll} with $\ell = 0$, $\Omega_1$ and now with $q(r)$ we find that 
		\begin{align*}
			\|\mu_{r}\|_{\Lrm^{q(r)}} & \le (1 + \eps) \big( \|\phi_{r}\sigma\|_{\Wrm^{-k,q(r)}} +|P_0\nu|(\Omega)+ \|[\Acal,\phi_{r}](\mu_{r-1})\|_{\Wrm^{-k,q(r)}} \big)\\
			& \le C\lambda^k (1 + \eps)\big( |\sigma|(\Omega) +|P_0\nu|(\Omega)+ \|\mu_{r-1}\|_{\Lrm^{q(r-1)}} \big).
		\end{align*}
		Thus, the $r$'th iteration can be estimated as ($1 \le r \le k-1$)
		\[
		\|\mu_r\|_{\Lrm^{q(r)}(\Omega)} \le C r \lambda^{rk}(1 + \eps)^r \big(|\sigma|(\Omega) + |\mu|(\Omega)\big).
		\]
		This completes the inductive step.
		
		For $r = k-1$ we then have shown (recall that $\eps < 1$)
		\begin{align*}
			\|\mu\|_{\Lrm^{p}(\Omega')} & \lesssim_{p,d,R}  \|\mu_{k-1}\|_{\Lrm^{q(k-1)}(\Omega)} \\
			& \le C (p,d,\Acal,L,\Omega') k \lambda^{k^2}  \big( |\sigma|(\Omega) + |\mu|(\Omega) \big).
		\end{align*}
		This finishes the proof. 
	\end{proof}
	

	\begin{proof}[Proof of Corollary~\ref{rem:Afree}]
		The idea is that the regularity of $\sigma$ acts as a barrier for the regularity of $\mu$ in the sense that one can only expect (for $k < d$)
		\[
		|\sigma|(\Omega) < \infty \qquad \Longrightarrow \qquad \|\mu\|_{\Lrm^p(\Omega')} < \infty \qquad \text{for $p < d/(d-k)$.}
		\]
		In the absence of the source term $\sigma$ (when $\mu$ is $\Acal$-free), there is no such obstruction to improve the integrability of $\mu$.
		This can be easily circumvented by artificially constructing a higher-order PDE constraint for $\mu$. The operator $\Delta^r_\Acal$ (from $V$ to $V$) associated to the symbol $\xi \mapsto [\Abb(\xi)^* \Abb(\xi)]^{2^{r-1}}$ has order $2^{r}k$ and satisfies the following properties: Firstly, every $\Acal$-free measure is a $\Delta_\Acal^r$-free measure since
		\[
		\Delta^r_\Acal (\mu) = \Delta_\Acal^{r-1} \circ \Delta_\Acal^{r-1} (\mu) = \cdots = \Delta^{r-1}_\Acal \circ \Delta^{r-2}_\Acal \circ \cdots \circ \Acal^* (\Acal \mu) = 0.
		\]
		Secondly, the wave-cones of $\Lambda_{\Delta_\Acal^r}$ and $\Lambda_\Acal$ coincide due to the identity
		\[
		\ker [\Abb(\xi)^* \Abb(\xi)]^{2^{r-1}} = \cdots = \ker [\Abb(\xi)^* \Abb(\xi)] = \ker \Abb(\xi).
		\] 	
		Thus, upon taking $r = r(k,d) > \log_2 (d/k)$, one may apply the estimates from 
		Theorem~\ref{thm:Lp} to $\Delta_\Acal^r$ and $\mu$. The smallness constant for $\Delta_\Acal^r$ still depends on the same parameters $p,d,\Acal,L, \Omega'$ but now they also depend on $r$ and therefore the smallness constant may not coincide with the one for $\triangle_\Bcal$. This however does not require to introduce a dependency parameter on $k$ since $k = k(\Acal)$.
	\end{proof}

	\begin{proof}[Proof of Theorem~\ref{thm:limiting}] 
		The argument is exactly the same as the one used in the proof  of Theorem~\ref{thm:Lp}, with the exception that the constant-rank and canceling assumptions will allows us to improve the regularity of the source term $\sigma$, namely that (recall that $\supp \mu_1 \Subset \Omega_1$ and $\dist(\Omega_1,\partial \Omega) \approx \lambda$)
		\begin{equation}\label{eq:lorraine}
			\|\Acal \mu_1\|_{\Wrm^{-1,d/(d-1)}} \le C\big( |\mu|(\Omega) + |\sigma|(\Omega)\big)
		\end{equation}
		with a constant $C = C(d,\Acal,L,\Omega',\lambda)$. Once this is established, the rest follows as before since 
		\[
		q(k-1) = \frac{d}{d-k} \qquad \text{for \, $q = \frac{d}{d-1}$}.
		\] 
		
		To see that the estimate in~\eqref{eq:lorraine} holds, 
		let $1 < p < d/(d-k)$. Observe that if $\eps$ is sufficiently small, by the previous part of the proof we may assume without any loss of generality that $\mu \in \Lrm^p_\loc(\Omega)$ (with bounds that depend solely on the total variation of $\mu$ and $\sigma$). Let $\chi$ be a suitable cut-off function on $\Omega$ satisfying $\chi \equiv 1$ on $\supp \phi_1$. Then, 
		\[
		\|[\Acal,\phi_1](\mu)\|_{\Wrm^{-k,d/(d-1)}} \lesssim \|[\Acal,\phi_1](\chi\mu)\|_{\Wrm^{-k,p(1)}} {\lesssim} \|\chi\mu\|_{\Wrm^{-1,p(1)}} \le C|\mu|(\Omega),
		\]
		where in the first inequality we have used that $p > 1$ and in the last inequality that $p < d/(d-1)$. Here, $\chi$ can be chosen so that the constant $C$ depends on $d,\Acal,\Omega',\lambda$. Similarly, we also get (recall that $k < d$)
		\[
		\|\phi_1 \sigma\|_{\Wrm^{-k,d/(d-k)}} \le C \|\phi_1 \sigma\|_{\Wrm^{-1,d/(d-1)}} \le C\big( |\mu|(\Omega) + |\sigma|(\Omega)\big),
		\] 
		where the last inequality is a direct consequence of the local estimates for canceling operators given in Theorem~\ref{thm:localVS}.
		Since $\Acal \mu_1 = \phi_1 \sigma + [\Acal,\phi_1](\chi\mu)$, this proves~\eqref{eq:lorraine}.
	\end{proof}

	\begin{proof}[Proof of Corollary~\ref{cor:CC}.]
		The proof of the first part of this corollary is immediate from Theorem~\ref{thm:Lp} since $\Lrm^p$-bounds for $p > 1$ imply $\Lrm^q$-equiintegrability for $1 \le q < p$. For the second part, we first observe that $\Lrm^q$-equiintegrability upgrades weak* convergence in the sense of measures to weak convergence locally in $\Lrm^q$. Further, if $\mu_j \to \mu$ in measure or almost everywhere (as maps), then the $\Lrm^q$-equiintegrability implies the strong convergence locally in $\Lrm^q$ by Vitali's convergence theorem.
	\end{proof}

	\begin{proof}[Proof of Proposition~\ref{prop:rigid}]
		Note first that by the assumed convexity of $\Kcal$ it holds that either the cone is acute, i.e., $-\Kcal \cap \Kcal = \{0\}$ or that $\Kcal = V$. In the latter case, $\Lambda_\Acal = \{0\}$ and $\Acal$ is elliptic. Thus, there are no non-zero \emph{integrable} solutions of our system. 
		
		To show the theorem also in the first case, we argue by contradiction. So assume that there is a non-zero $u \in \Lrm^1(\R^d;V)$ with $\Acal u = 0$ and $u \in \Kcal$ almost everywhere.
		By the linearity of the integral, it holds that $\Kcal \ni e := \int u \dd x \neq 0$; here we have used that $\Kcal$ is an acute cone, so that $\int u \dd x = 0$ if and only if $u$ is identically zero. 
		
		We define $u_\lambda(x) := \lambda^{-d} u(x / \lambda)$ and proceed to apply a well-known rescaling argument (see~\cite[Proposition~2.2]{VanSchaftingen13}): Since $u \in \Lrm^1(\R^d;V)$, one has that $u_\lambda \to e\delta_0$ in the sense of distributions as $\lambda \to 0^+$. It follows that 
		\[
		\Acal (e\delta_0) = 0 \qquad \text{in $\Dcal'(\R^d;V)$.}
		\]
		Applying the Fourier transform to both sides of this equation, we obtain $\Abb(\xi)e = 0$ for all $\xi \in \R^d$. Hence $e \in \Lambda_\Acal$. As $e \in \Kcal \setminus \{0\}$, the the assumption $\Kcal \cap \Lambda_\Acal = \{0\}$ yields the sought contradiction.
		
		Remark~\ref{rem:rigid} follows by observing that in reaching a contradiction it suffices to assume that
		\[
		\Kcal \cap \bigcap_{|\xi| = 1} \ker \Abb(\xi) = \{0\}.
		\]
		This completes the proof.
	\end{proof}

	\section{Counterexamples} \label{sc:counterexamples}
	
	As shown in the previous sections, we can expect compensated compactness in $\Lrm^1_\loc$ under the assumption that an $\Acal$-free (or uniformly $\Acal$-bounded) sequence satisfies
\[
	\dist(\mu_j,\Dcal) \le \eps, \qquad 0 < \eps \ll 1,
\] 
where $\Dcal \subset L$ is a one-sided cone and $L$ is a subspace with no $\Lambda_\Acal$-connections. However, relaxing this to a double-sided constraint 
\begin{align}\label{eq:double_p}
	\dist(\mu_j,\Dcal \cup -\Dcal) \to 0 \quad \text{in $\Lrm^p$ for any $p \in [1,\infty]$}
\end{align}
does not yield analogous compactness results.

We will first demonstrate,  with three counterexamples {and a general result about the failure of $\Lrm^1$-compensated for $L$-elliptic systems from~\cite{ArroyoRabasa19?}}, that~\eqref{eq:double_p} with $p = 1$ is not enough to rule out concentrations. 

For the following three examples, we specialize to the case $\Acal = \curl$ for $M^{n \times d}$-valued maps, for which we have $\Lambda_{\curl} = \setn{ A \in \R^{n \times d} }{ \rank A \leq 1 }$. The first two examples are constructive but less general, and the third and fourth examples are rather abstract but cover general situations.  

\begin{example}[Gradients swirling around an elliptic space]
Let $\SD(d) \subset \R^{d \times d}$ denote the space of trace-free symmetric $(d \times d)$-matrices ($d \geq 2$). Clearly, this space only intersects the cone of rank-one matrices at the origin. We will construct a sequence of gradients $\nabla v_j \in \Lrm^1(B_1;\R^{d\times d})$  satisfying the following properties:
\begin{enumerate}[(i)]
    \item\label{itm:L1_dist} $\|\dist( \nabla v_j, \SD(d))\|_{\Lrm^1} \to 0$ as $j \to \infty$;
    \item $\nabla v_j \Lcal^d \toweakstar 0$ in $\Dcal'(B_1;\R^{d\times d})$;
    \item\label{itm:L1_lim} $\nabla v_j \not\to 0$ in $\Lrm^1(B_1;\R^{d\times d})$. 
\end{enumerate}
This means that imposing the $\Lrm^1$-closeness condition~\eqref{itm:L1_dist} to a $\frac{(d+2)(d-1)}{2}$-dimensional subspace (this subspace is $2$-dimensional if $d = 2$) that only intersects the rank-one cone at the origin, is insufficient for improving weak* convergence to strong $\Lrm^1$ convergence for a curl-free sequence, in opposition to the statement of Corollary~\ref{cor:CC} for $\Lrm^\infty$-closeness to a fixed polar.

Since the ideas are similar, we will only construct our example for $d = 2$. Let $0 < \eps < 1$ and define 
\[
  u_\eps(x) :=\frac{1}{\abs{\ln \eps}} g_\eps(\ln \abs{x}),  \qquad\text{where}\qquad
  g_\eps(t) := \eta\biggl( \frac{t}{\ln \eps} \biggr) t,
\]
and $\eta \colon \R \to \R$ is a smooth function with $\eta(t) = 1$ for $\eta \leq 1$ and $\eta(t) = 0$ if and only if $t \geq 2$. We compute, where we set $R_\eps(x) := \frac{\ln \abs{x}}{\ln \eps}$,
\begin{align*}
  \nabla^2 u_\eps(x) &=\frac{1}{\abs{\ln \eps}} \bigl( g_\eps'(\ln \abs{x}) \nabla^2 \ln \abs{x} + g_\eps''(\ln \abs{x}) \nabla \ln \abs{x} \otimes \nabla \ln \abs{x} \bigr) \\
  &= \frac{1}{\abs{\ln \eps}} \bigl(\eta(R_\eps(x)) + \eta'(R_\eps(x))R_\eps(x)\bigr) \nabla^2 \ln \abs{x} \\
  &\qquad + \frac{1}{\abs{\ln \eps}^2} \bigl(2\eta'(R_\eps(x)) + \eta''(R_\eps(x))R_\eps(x) \bigr) \frac{x}{\abs{x}^2} \otimes \frac{x}{\abs{x}^2} \\
  &=: \mathrm{I}_\eps + \mathrm{II}_\eps.
\end{align*}
We observe that $R_\eps(x) \leq 1$ for $\abs{x} \geq \eps$ and $R_\eps(x) \geq 2$ for $\abs{x} \leq \eps^2$. Notice also that $\eta(t) + \eta'(t)t$ is $1$ if $t \leq 1$ and $0$ if $t \geq 2$ and bounded and nonzero in between. Thus,
\[
  \int_{B_1} \abs{\mathrm{I}_\eps} \dd x
  = \int_{B_1 \setminus B_{\eps^2}} \abs{\mathrm{I}_\eps} \dd x,
\]
whereby
\[
  \int_{B_1} \abs{\mathrm{I}_\eps} \dd x 
  \approx \frac{1}{\abs{\ln \eps}} \int_{B_1 \setminus B_{\eps^2}} \abs{\nabla^2 \ln \abs{x}} \dd x
  = \frac{2\sqrt{2}\pi}{\abs{\ln \eps}}\int_{\eps^2}^1 \frac{1}{r} \dd r
  = 4\sqrt{2}\pi .
\]
Moreover, $2\eta'(t) + \eta''(t)t$ is supported in $(1,2)$ and bounded, we have
\[
  \int_{B_1} \abs{\mathrm{II}_\eps} \dd x
  \lesssim \frac{1}{\abs{\ln \eps}^2} \int_{B_1 \setminus B_{\eps^2}} \abs{x}^{-2} \dd x \to 0.
\]

Now define $v_\eps := \nabla u_\eps$, for which $\nabla v_\eps = \nabla^2 u_\eps$. Then, since $g_\eps$ converges to the function that is $1$ almost everywhere in $\Lrm^1_\loc$  this yields~(ii). Moreover, $\int_{B_1} \abs{\nabla^2 u_\eps} \dd x \approx 1$, so~(iii) also holds. For~(i), we note that $\nabla^2 \ln \abs{x} \in \SD(2)$ almost everywhere since $\nabla^2  \ln \abs{x}$ is a Hessian and $\trace (\nabla^2 \ln \abs{x}) = \Delta \ln \abs{x} = 0$ (recall that $\ln \abs{x}$ is the fundamental solution of the Laplacian in two dimensions). Therefore,
\[
  \dist(\nabla v_j(x),\SD(2)) \leq  \abs{\mathrm{II}_\eps}.
\]
Thus,
\[
  \int_{B_1} \dist(\nabla v_j(x),\SD(2)) \dd x
  \leq  \int_{B_1} \abs{\mathrm{II}_\eps} \dd x \to 0,
\]
yielding~(i).
\end{example}

\begin{example}[Quasiconformal maps] We recall that the conformal coordinates $(a^+,a^-) \in \C^2$ of a matrix $A\in \R^{2 \times 2}$ are defined by the rule
\[
	Av = a_+ v + a_-\cl{v},
\]
under the identification of vector $v =(x,y) \in \R^2$ with complex numbers $v = x + \mathrm{i}y$. Verifying that this defines a linear isomorphism from $\R^{2 \times 2}$ onto $\C^2$ is straightforward. The complex dilation of a matrix $A \in \R^{2 \times 2}$, which is defined as
\[
	\mu_A\coloneqq \frac{a_-}{\cl{a_+}},
\]
quantifies how far is a matrix $A$ from being conformal (precisely when $a_- = 0$). In this vein, for a real number $r \in \R$ we define
\[
E_r \coloneqq \set{A \in M^{2 \times 2}}{\mu_A = r},
\]
of matrices with complex dilation equal to $r$. With these conventions, it is straightforward that 
\[
	E_0 = \set{\begin{pmatrix}
			a & b \\-b & a
		\end{pmatrix}}{a,b \in \R^2}
\]
is the space of conformal matrices. 
%
Notice that the set space of conformal matrices possesses no rank-1 connections.  
Indeed, this follows from the fact that $\det(A) = \frac 12|A|^2$ for all $A \in E_0$. Moreover, since the change to conformal coordinates is a linear isomorphism, it follows that 
\[
	A \in E_r \quad \Longrightarrow \quad \dist(A,E_0) \le C |a_-| \le \frac{C|A|\cdot|r|}{\sqrt{2}}.
\]
Hence,
\begin{equation}\label{eq:conestocones}
	\dist\left(\frac{A}{|A|},E_0\right) \le C |r| \quad \text{for all $A \in E_r$}
\end{equation}
and in particular $E_r$ belongs to a closed (two-sided) $Cr$-neighborhood of $E_0$.

\begin{theorem}[Astala et al.~{\cite[Thm. 3.18]{Astala08}}]\label{thm:astala} Let $\eps \in (0,1)$. For bounded set $\Omega \subset \R^2$ there exists a mapping $f \in W^{1,1}(\Omega;\R^2) \cap C(\bar\Omega;\R^2)$ satisfying
\begin{enumerate}
	\item $f(x) = 0$ on $\partial \Omega$,
	\item $\nabla f(x) \in E_{-\eps} \cup E_{\eps}$ for almost every $x \in \Omega$,
	\item $f \in W^{1,q}(\Omega;\R^2)$ for all $q < 1 + \eps$, but for any ball $B \subset \Omega$ it holds
	\[
		\int_B |\nabla f|^{1 + \eps} = \infty
	\]
\end{enumerate}
\end{theorem}
The following result, which follows directly from the Theorem~\ref{thm:astala} above, shows that the higher-integrability bounds from Theorem~\ref{thm:Lp} may fail for non-convex constraints of the polar in the case $d = 2$. 
\begin{proposition}[Almost conformal curl-free fields]\label{prop:astala} For every $p \in (1,2)$ and every $\eps \in (0,p -1)$, there exists an integrable matrix field $\mu \in \Lrm^{q}(\Omega;\R^2)$ for all $q \in [1,1 + \eps)$, such that
	\[
		\curl \mu = 0 \quad \text{on $\Omega$}
	\]
	and 
	\[
		\dist\left(\frac{\mu}{|\mu|}(x),E_0\right) \le C\eps.
	\]
	However, 
	\[
		\int_B |\mu|^{1 + \eps} \, \dd x = \infty \quad \text{for every ball $B \subset \Omega$.}
	\] 
\end{proposition}
\begin{proof}
	Simply let $\mu = \nabla f$, where $f$ is the function given by Theorem~\ref{thm:astala} for $\eps \in (0,p-1)$. Clearly $\mu$ is curl-free. Moreover, the estimate~\eqref{eq:conestocones} and the fact that $\nabla f(x) \in E_{-\eps} \cup E_{\eps}$ imply that distance from the polar of $\mu$ to $E_0$ is bounded by $\eps$ up to a constant. The failure of the local $p$-integrability follows from (3)  in the theorem above, the fact that $p > 1 + \eps$ and H\"older's inequality.
\end{proof}
\end{example}

Now we show that under the (weaker) assumption of $\Lrm^1$-closeness of a sequence of gradients to a \textit{line} away from the wave cone, we still cannot expect compensated compactness. However, this example is non-constructive, and relies on the characterization~\cite[Theorem~1]{KristensenRindler10YMNoErr} of $\BV$-Young measures via a Jensen-type inequality. We refer the reader to~\cite[Chapter~12]{Rindler18book} for notation.

\begin{example}
    Fix any $P_0 \in \R^{\ell} \otimes \R^d$ with $|P_0| = 1$, $\rank P_0 \geq 2$, and consider the following generalized Young measure:
	\[
        \boldsymbol{\nu} = (\nu_x,\lambda_\nu,\nu_x^\infty) := \biggl(\delta_0,\delta_0,\frac{1}{2}\delta_{P_0} + \frac{1}{2}\delta_{-P_0}\biggr) \in \Ybf^{\Mcal}(B_1;\R^{\ell} \otimes \R^d).
	\]
	Namely, $\boldsymbol{\nu}$ concentrates only at the origin with the direction of concentration oscillating between $\pm P_0$. We have for the barycenter
	\[
        [\boldsymbol{\nu}] := [\delta_0]\Lcal^d\mres B_1 + \Big[\frac{1}{2}\delta_{P_0} + \frac{1}{2}\delta_{-P_0}\Big]\delta_0 = 0
	\]
Also $\lambda_\nu(\partial B_1) = 0$. Now take any quasiconvex $h \colon\R^{\ell} \otimes \R^d\to \R$ with linear growth. We have
\[
    h\biggl([\delta_0] + \Big[\frac{1}{2}\delta_{P_0} + \frac{1}{2}\delta_{-P_0}\Big]\frac{\di\delta_0}{\di\Lcal^d}(x)\biggr) = h(0)  = \dprb{ h, \delta_0 } + \dprBB{ h^\#, \frac{1}{2}\delta_{P_0} + \frac{1}{2}\delta_{-P_0} } \frac{\di\delta_0}{\di\Lcal^d}(x),
\]
for almost every $x$, where
\[
    h^\#(A) := \lim_{\eps \to 0^+} \, \sup \, \setBB{ \frac{f(tB)}{t} }{ 0 < |A-B| < \eps, \ t > \frac{1}{\eps} }
\]
is the upper (generalized) recession function associated to $h$. Hence, by~\cite[Theorem~1]{KristensenRindler10YMNoErr}, we deduce that $\boldsymbol{\nu} \in \BVY(B_1;\R^{\ell} \otimes \R^d)$, so by~\cite[Proposition~12.18]{Rindler18book}, we can find an (improved) generating sequence of maps $(u_j)_j \subset (\Wrm^{1,1}\cap\Crm^\infty)(B_1;\R^\ell)$ with $\nabla u_j \overset{\textbf{Y}}{\to} \boldsymbol{\nu}$. But if we choose
\[
    f(x,A) := \dist(A, \spn\{P_0\}) \in \textbf{E}(B_1;\R^{\ell} \otimes \R^d),
\]
where $\textbf{E}(B_1;\R^{\ell} \otimes \R^d)$ is the pre-dual of $\textbf{Y}^{\Mcal}(B_1;\R^{\ell} \otimes \R^d)$, we have that
\[
    \ddprb{ f, \delta [\nabla u_j] } = \int_{B_1} \dist\biggl(\nabla u_j(x), \spn\{P_0\}\biggr) \dd x \to 0.
\]
So we have found a sequence of gradients (curl-free functions) that converge weakly* in $\Mcal(B_1;\R^{\ell} \otimes \R^d)$ to $[\boldsymbol{\nu}] = 0$, and the distance between their polars and $\{P_0, -P_0\}$ converges to 0 strongly in $\Lrm^1$. However, we cannot improve weak* convergence to strong convergence here since~\cite[Corollary~12.15]{Rindler18book} tells us that this is only possible provided $\lambda_\nu = 0$, which is not true.
\end{example}

More generally, for operators $\Acal$ satisfying the constant-rank property, we recall the following abstract result on the failure of $\Lrm^1$-compensated compactness for $\Acal$-free measures~\cite[Corollary~3.1]{ArroyoRabasa19?}:
\begin{proposition} Let $\Acal$ be a linear PDE operator from $V$ to $W$ that satisfies the constant-rank property.
	Let $L$ be a non-trivial subspace of 
	\[
		\spn \Lambda_{\Acal} \subset V,
	\]  and assume that $L$ has no non-trivial $\Lambda_\Acal$-connections, i.e., 
	\[
	L \cap \ker\Abb(\xi) = \{0\}  \qquad \text{for all $\xi \in \R^d \setminus \{0\}$}.
	\]
	Then, for every non-trivial cone $\Ccal \subset L$, there exists a sequence $\{w_j\} \subset \Crm^\infty(B_1;V)$ of $\Acal$-free vector-valued functions satisfying
	\begin{align*}
		w_j  \, \mathcal L^d & \; \toweakstar \; 0 \quad \text{in $\Mcal(B_1;V)$,}\\
		\dist(w_j,\Ccal \cup - \Ccal) & \; \to  \; 0 \quad \text{in $\Lrm^1(B_1)$},
	\end{align*}
	but
	\begin{gather*}
		w_j \not\toweak 0 \quad \text{in } \Lrm^1(B_1;V), \text{ and}  \\
		\{|w_j|\} \; \text{is not locally equiintegrable on any sub-domain of $B_1$}.
	\end{gather*}
\end{proposition}

\begin{remark}
The proof in this case appeals to the  characterization of $\Acal$-free Young measures, which has in been recently established independently in~\cite{ArroyoRabasa19?} and~\cite{KristensenRaita19?}. In this regard,  a thorough discussion and counterexamples to the compensated compactness for $\Lrm^1$-asymptotically $L$-elliptic systems can be found in~\cite[Sec. 3.2]{ArroyoRabasa19?}.
\end{remark}

Finally, we show that in our compensated compactness result, Corollary~\ref{cor:CC}, we cannot expect strong convergence since \emph{bounded oscillations} are still allowed.

\begin{example}[Laminates]\label{ex:laminates}
	Let $P \in \Lambda_\Acal$ and let $\xi \in \R^d$ be a normal direction such that $P \in \ker \Abb(\xi)$. If $\Ccal \subset E$ is a set with non-empty interior, then we may find $\delta > 0$ and vectors $A,B \in \Ccal$ such that
	\[
		A = B + \delta P.
	\]  
	Let $\phi:\R\to \R$ be the $[0,1]$-periodic extension of the indicator function $\chi_{[0,\frac 12)}$ and consider the sequence of $\{0,1\}$-valued maps
	\[
		\phi_j(x) := \phi(j(x \cdot \xi)), \qquad x \in \R^d.
	\]
	We define an $\{A,B\}$-valued sequence of laminates $u_j : \R^d \to \R^m$ oscillating on the $\xi$-direction by letting
	\[
		u_j:= B + \delta P \phi_j.
	\]
	Notice that this conforms a sequence of $\Acal$-free measures since
	\[
		\Acal u_j = [j\delta \phi'(j(\mathrm{d}x \cdot \xi))] \Abb(\xi)[P] = 0 \qquad \text{in $\Dcal'(\R^d;V)$.}
	\]
	By construction it holds
	\begin{align*}
		u_j &\, \toweak\,  \frac{A + B}2 \quad \text{locally in $\Lrm^p$ for all $p \in [1,\infty)$},\\
		u_j &\, \not\to\,  \frac{A + B}2 \quad  \text{locally in $\Lrm^1$.}\\
	\end{align*}
%
%
\end{example}

\end{document}